\documentclass[12pt]{amsart}
\usepackage[total={6in,9in},
top=1in, left=1in, right=1in, bottom=1in]{geometry}
\usepackage{amsmath}
\usepackage[english]{babel}
\usepackage{amsfonts}
\usepackage{amssymb}
\usepackage{dsfont}
\usepackage{mathrsfs}
\usepackage{graphicx}
\usepackage{fancyhdr}
\usepackage{multicol}
\usepackage{enumitem}
\usepackage{amsthm}
\usepackage{empheq}
\usepackage{cases}
\usepackage[all]{xy}
\usepackage{stmaryrd}
\usepackage[colorlinks, citecolor=blue, linkcolor=red]{hyperref}
\usepackage{mathrsfs}
\usepackage{tikz,tikz-cd}
\tikzset{
	subset/.style={
		draw=none,
		edge node={node [sloped, allow upside down, auto=false]{$\subset$}}},
	Subset/.style={
		draw=none,
		every to/.append style={
			edge node={node [sloped, allow upside down, auto=false]{$\subset$}}}
	}
}
\tikzset{
	labl/.style={anchor=south, rotate=90, inner sep=.50mm}
}

\allowdisplaybreaks[1]

\newcommand{\erre}{\mathds{R}}

\newcommand{\ricc}{\operatorname{Ric}}
\newcommand{\weyl}{\operatorname{W}}
\newcommand{\diver}{\operatorname{div}}

\newcommand{\KN}{\mathbin{\bigcirc\mspace{-15mu}\wedge\mspace{3mu}}}

\newcommand{\ra}{\rightarrow}

\newcommand{\Longra}{\Longrightarrow}

\newcommand{\set}[1]{{\left\{#1\right\}}}               
\newcommand{\pa}[1]{{\left(#1\right)}}                  
\newcommand{\sq}[1]{{\left[#1\right]}}                  
\newcommand{\abs}[1]{{\left|#1\right|}}                 

\newcommand{\eps}{\varepsilon}                           

\newcommand{\ol}[1]{\overline{#1}}
\renewcommand{\hat}[1]{\widehat{#1}}
\renewcommand{\tilde}[1]{\widetilde{#1}}

\newcommand{\riem}{\mathrm{Riem}}

\newcommand{\ric}{\mathrm{Ric}}

\newcommand{\boc}{\operatorname{B}}

\newcommand{\pyamw}{\frac{||W||_{L^{p}}}{Y(M,[g])}}

\newtheorem{theorem}{\textbf{Theorem}}[section]

\newtheorem{cor}[theorem]{\textbf{Corollary}}
\newtheorem{defi}[theorem]{\textbf{Definition}}
\theoremstyle{remark}
\newtheorem{rem}[theorem]{\textbf{Remark}}

\numberwithin{equation}{section}
\title[]
{Some isolation and stability results for Einstein manifolds}

\author[L. Branca]{Letizia Branca}
\address[Letizia Branca]{Dipartimento di Matematica, Universit\`{a} degli Studi di Milano, Via Saldini 50, 20133, Milano, Italy.}
\email[]{letizia.branca@unimi.it}
\author[K. Kr\"{o}ncke]{Klaus Kr\"{o}ncke}
\address[Klaus Kr\"{o}ncke]{Institutionen f\"{o}r Matematik, KTH Stockholm, Lindstedtsv\"{a}gen 25, 10044
	Stockholm, Sweden}
\email[]{kroncke@kth.se}

\date{}

\subjclass[2010]{}

\begin{document}
	\maketitle
	\date{\today}
	\noindent
	\begin{abstract} 
		We prove new isolation and stability results for Einstein manifolds in a variety of settings. Imposing conditions on the Weyl tensor, we establish new stability criteria for compact, asymptotically hyperbolic (AH) and asymptotically locally Euclidean (ALE) manifolds and an isolation result in the latter setting. For compact K\"{a}hler and Sasaki $\eta$-Einstein manifolds, we provide similar results which involve the Bochner tensor and the contact Bochner tensor, respectively.		
	\end{abstract}
	\section{Introduction and main results}
	One of the most important problems in Riemannian geometry and Geometric analysis is the study of the so-called \emph{canonical} (or \emph{special}) metrics. Many of these metrics are related to curvature conditions. Indeed, a well known example is given by \emph{Einstein} metrics, which are those metrics whose Ricci tensor (defined as the trace of the Riemann curvature tensor) is proportional to the metric itself, namely
	\[\ric=\lambda g,\]
	$\lambda \in \erre$. These metrics arise from physics, since they are solutions to the vacuum Einstein field equations; it is known that when $n=2,3$, a manifold is Einstein if and only if it has constant sectional curvature. In dimension four
	the existence of Einstein metrics requires some strong topological conditions, indeed on a compact four-dimensional Einstein manifold the Gauss-Bonnet formula forces the Euler characteristic to be non-negative and combined with the signature formula it guarantees the validity of the Hitchin–Thorpe inequality, 
	\[\chi(M)\geq\frac{3}{2}\abs{\tau(M)},\]
	where $\chi(M)$ and $\tau(M)$ are the Euler characteristic and the signature of $M$, respectively. Note that also in the non-compact case, with specified
	asymptotic geometry at infinity, we have the validity of an analogous inequality (\cite{DaiWei}). We point out that for higher dimensional Einstein manifolds, to date, not much is known. When $n\geq3$
	the Riemann curvature tensor $\operatorname{Riem}$ admits the
	decomposition:
	\[
	\operatorname{Riem}=\weyl+\dfrac{1}{n-2}
	\ricc\KN g
	-\dfrac{S}{2(n-1)(n-2)}g\KN g,
	\]
	where $\weyl$ is the \emph{Weyl} tensor,
	$\ricc$ is the \emph{Ricci tensor} and
	$S$ is the \emph{scalar curvature}. The Weyl tensor is conformally invariant and it measures the conformal flatness. In \cite{KobayashyConf}, Kobayashi introduced a conformally invariant functional on the space of metrics of a compact manifold of dimension $n$, namely the Weyl functional, defined as
	\[\mathfrak{W}(g):=\int_M\abs{\weyl}^{\frac{n}{2}}dV_g.\]
	It has been shown that, on four-dimensional compact Riemannian manifolds, the Weyl functional satisfies some interesting topological lower bounds: Gursky (\cite{Gursky98}) proved that when the Yamabe invariant of $M$ is positive and the intersection form has at least one positive eigenvalue, the self-dual part of the Weyl functional is greater than or equal to a quantity depending on $\chi(M)$ and $\tau(M)$. A similar result was obtained in the case of negative Yamabe invariant, provided the existence of a non-trivial conformal vector field (\cite{Gursky99}); later, LeBrun proved, without any assumption on the sign of the Yamabe invariant, that the same inequality was obtained by conformal classes of symplectic type on a Del Pezzo surface (\cite{LeBrun2015}). Note that, while all these results were obtained on compact Riemannian manifolds, some rigidity results have been proved on compact manifolds of dimension four with boundary (see, for instance \cite{CatinoNdiaye}). 
	As far as higher
	dimensional cases are concerned, many authors managed to prove that, on an Einstein manifold with positive scalar curvature, if the Weyl functional satisfies a pinching condition (\cite{BCDM,CMbook,FX17,HebeyVaugon,Singer}) then $(M,g)$ is a quotient of the standard sphere.
	In dimension $n=4$, an improvement of this result is due to Gursky and LeBrun (\cite{Gursky200,GurskyLeBrun}), namely on a four-dimensional Einstein manifold with positive scalar curvature if
	\[\int_M\abs{\weyl}^2dV_g\leq\frac{1}{6}\int_M S^2dV_g,\]
	then $(M,g)$ is isometric to $\mathbb{S}^4$, $\erre\mathbb{P}^4$ or $\mathbb{CP}^2$ with their standard metric. The result is sharp, since equality is achieved by $\mathbb{CP}^2$ endowed with the Fubini-Study metric. 
	We point out that the optimal result was proven by Gursky and LeBrun (\cite{Gursky200,GurskyLeBrun}) exploiting the peculiarity of dimension $4$ and making use of a modified Yamabe invariant: they proved that
	given $(M^4,g)$ an Einstein manifold with positive Yamabe invariant and $\abs{\weyl}\neq0$, 
	\[\int_M\abs{\weyl^+}^2dV_g\geq\frac{1}{6}\int_M S^2dV_g,\]
	and equality holds if and only if $\nabla \weyl^+=0$ (note that the same result holds if we replace $\weyl^+$ with $\weyl^-$). Moreover both $\mathbb{CP}^2$ with the standard orientation and the Fubini-Study metric and $\mathbb{S}^2\times\mathbb{S}^2$ with the product metric realize the equality. More precisely, by Cartan's classification of irreducible symmetric spaces (\cite{cartan1,cartan2}) equality is obtained only by these two spaces, up to quotients. In \cite{BCDM}, optimal pinching results for compact $n$-dimensional Einstein manifolds with positive Yamabe invariant were obtained adapting the technique due to Gursky and LeBrun.
	\\
	It seems thus natural to ask whether it is possible to find classification results also in the non-compact case. In particular, one of the aims of this article is to provide isolation results for the Weyl functional on some specific classes of non-compact Einstein manifolds, more precisely: Ricci-flat \emph{asymptotically locally Euclidean} (ALE) manifolds.
	
Before we state the first theorem, let us remark that every odd-dimensional Ricci-flat ALE manifold is flat. This is because in that case, the fundamental group at infinity is trivial and flatness follows from the Bishop-Gromov volume comparison theorem.		
	Our first result is as follows:
		\begin{theorem}\label{isolation thm intro}
		Let $(M,g)$ be a Ricci-flat ALE manifold. Then, either $(M,g)$ is flat or,
		\begin{align*}
			||\weyl||_{L^\frac{n}{2}}\geq B(n)\set{\frac{1}{\sqrt{\pi n(n-2)}}\pa{\frac{\Gamma(n)}{\Gamma\pa{\frac{n}{2}}}}^{\frac{1}{n}}\abs{\Gamma}^{\frac{1}{n}}}^{-2}C(n)^{-1},
		\end{align*}
where $\abs{\Gamma}$ denotes the order of the group acting at infinity (see Definition \ref{d-ale} for more details) and the constants $B(n)$ and $C(n)$ are defined as
\begin{align*}
B(4)&=\frac{4}{3},\qquad B(n)=\frac{1}{2\pa{2-\frac{n+1}{n-1}}} \text{ for }n\geq 6,\\
C(4)=\frac{\sqrt{6}}{4},\qquad
C(6)&=\frac{\sqrt{70}}{2\sqrt{3}},\qquad C(n)=2+\frac{1}{2}\frac{n(n-1)-4}{\sqrt{n(n-1)(n+1)(n-2)}},\text{ for }n\geq 8.
\end{align*}
	\end{theorem}
\begin{rem}	
A fundamental tool we use in this paper is the classical Bochner-Weitzenb\"{o}ck formula for the Weyl tensor, holding on every Einstein manifold; that is
	\[\frac{1}{2}\Delta \abs{\weyl}^2=\abs{\nabla\weyl}^2+\frac{2}{n}S\abs{\weyl}^2-2Q,\]
	where $Q$ is a cubic term in Weyl: $Q=
	2W_{pqrs}W_{ptru}W_{qtsu}+\frac{1}{2}
	W_{pqrs}W_{pqtu}W_{rstu}$. The constant $C(n)$ in the above theorem is such that $|Q|\leq C(n)|W|^3$. Note that due to that identity, $\weyl$ does not vanish on an open neighbourhood unless it is identically zero.
\end{rem}
	
	It is known that Einstein metrics on a compact manifold can also be regarded as critical points of the  Einstein-Hilbert functional
	\[\mathfrak{S}(g)=\int_M SdV_g.\] 
	when restricted to a subset of the set of smooth metrics consisting in metrics of constant volume $c$ ($\mathcal{M}_c\subset\mathcal{M}$).
	It is known that Einstein metrics are neither local maxima nor minima of the Einstein-Hilbert functional on $\mathcal{M}_c$ (\cite{Muto}). However, it is possible to introduce a notion of stability, which is strictly related to the so-called Einstein operator, $\Delta_E$. Many Einstein manifolds are known to be stable, such as the round sphere, the hyperbolic space and their quotients, the Euclidean space, $\mathbb{CP}^n$, Spin manifolds admitting a nonzero parallel spinor, K\"{a}hler-Einstein manifolds with non-positive scalar curvature and quaternion-K\"{a}hler manifolds of negative scalar curvature (\cite{CaoHe,DaiWangWei,DaiWangWei2007,Koiso80,kroencke2024stabilityscalarcurvaturerigidity,Wang91}). On the other hand, many examples of unstable Einstein manifolds have been constructed (see \cite{Bohm,GibbonsHartnoll,GibbonsHartnollPope,GubserMitra,PagePope1,PagePope2} and note that these examples are of positive scalar curvature). In the non-compact case there are examples of unstable Ricci-flat metrics and unstable Einstein manifolds with negative scalar curvature (\cite{Gross,HallHaslhoferSiepmann,Warnick}); in particular in \cite{Kroncke17}, the second author managed to provide new examples of stable and unstable Einstein metrics, \emph{via} a warped product construction. In \cite{Kroncke15}, some stability criteria involving the Weyl (and Bochner) curvature were obtained by the second author; we establish a stability criterion (involving the Weyl curvature) for two classes of manifolds in the non-compact case: ALE manifolds and \emph{asymptotically hyperbolic} (AH) manifolds which are Einstein (also known as \emph{Poincaré-Einstein} (PE) manifolds), respectively. Moreover, we extend the stability results in \cite{Kroncke15} for compact Einstein (and K\"{a}hler-Einstein) manifolds, providing a criterion involving the $L^p$, $p\in[\frac{n}{2},\infty]$, norm of the Weyl (Bochner) tensor.
	
	We point out that in stability criteria involving the Weyl tensor, we are going to use the following fact: the Weyl tensor defines a symmetric trace free operator on the space of symmetric two tensors,
	\[\mathscr{W}: \Gamma(S^2M)\ra \Gamma(S^2M), \quad h_{ik}\mapsto W_{ijkl}h_{ik}.\]
	This space has dimension $\frac{n(n+1)}{2}$, therefore denoting $\lambda$ the largest eigenvalue of $\weyl$, we have that for every symmetric $(0,2)$-tensor $h$ with vanishing divergence the following inequality holds: 
	\begin{align}\label{ineq weyl h point}
		W_{ijkl}h_{ik}h_{jl}\leq\lambda\abs{h}^2\leq s(n)\abs{\weyl}\abs{h}^2,
	\end{align}
	where
	\[s(n)=\sqrt{\frac{(n+2)(n-1)}{n(n+1)}}\]	
	(see \cite[Lemma 2.4]{Huisken} for more details).
	
Again for ALE manifolds, we prove the following:
	\begin{theorem}\label{stability thm intro}
		Let $(M,g)$ be a Ricci-flat ALE Riemannian manifold of dimension $n\geq 6$. If 
		\begin{align}\label{stability}
			||\weyl||_{L^{\frac{n}{2}}}\leq\frac{1}{2}\set{\frac{1}{\sqrt{\pi n(n-2)}}\pa{\frac{\Gamma(n)}{\Gamma\pa{\frac{n}{2}}}}^{\frac{1}{n}}\abs{\Gamma}^{\frac{1}{n}}}^{-2}s(n)^{-1},
		\end{align}
		then $g$ is stable. 
	\end{theorem}
	We point out that, thanks to the Gauss-Bonnet formula on even dimensional manifolds with boundary, when $n=6$, it is possible to deduce a stability criterion involving the Euler characteristic, the integral of $\mathrm{tr}(\weyl^3)$ and a boundary term (which we will denote as $I(\mathbb{S}^5/\Gamma)$, for more details see Subection \ref{s-alesix})
	\begin{theorem}\label{t-ec}
		Let $(M,g)$ be a Ricci-flat ALE Riemannian manifold of dimension $n=6$, if
		\begin{align}\label{t-eulc}
			6\int_M\mathrm{tr}(\weyl^3)dV_g+I(\mathbb{S}^5/\Gamma)-384\pi^3\chi(M)\leq \frac{1}{\abs{\Gamma}}\frac{2^2\cdot 3^3\cdot7^2}{5^3}\pi^3
		\end{align}
		then $(M,g)$ is stable. 
	\end{theorem}
	We define the Yamabe invariant of a complete manifold as
	\begin{align*}
		Y(M,[g])=\inf_{\stackrel{u\in C^{\infty}_0(M)}{u\neq 0}}\frac{\int_M\pa{\frac{4(n-1)}{n-2}\abs{\nabla u}^2+Su^2}dV_g}{\pa{\int_M\abs{u}^{\frac{2n}{n-2}}dV_g}^{\frac{n-2}{n}}}
	\end{align*}
	and we point out that this quantity is well defined since it is defined over compactly supported functions; moreover it is conformally invariant by definition.
	We then prove the following result for PE manifolds:
	\begin{theorem}\label{t-stability PE}
		Let $(M,g)$ be a PE manifold of dimension $n\geq4$ with positive Yamabe invariant. If there exists $p\in\sq{\frac{n}{2},\infty}$ such that
		\begin{equation}\label{stab Lp}
			||\weyl||_{L^p}\leq Y(M,[g])^{\frac{n}{2p}} \frac{(n-2)}{\sq{8(n-1)}^{\frac{n}{2p}}}s(n)^{-1}
		\end{equation}
		is satisfied, then $(M,g)$ is stable. If the inequality is strict, then $g$ is strictly stable.
	\end{theorem}

	\begin{rem}
Due to the perturbation result of Graham and Lee (\cite{GrahamLee}), every conformal class on $S^{n-1}$ which is sufficiently close to the one of the round metric is the conformal boundary of a PE metric on $\mathbb{R}^n$. These perturbed PE metrics are $C^{k,\alpha}$-close to the hyperbolic metric and their conformal backgrounds are $C^{k,\alpha}$-close to the flat metric on the disk. By conformal invariance of the Weyl tensor, the Weyl tensor of these PE metrics can have arbitrarily small $L^{n/2}$-norm. For this reason, we cannot obtain an analogue to Theorem
 \ref{isolation thm intro} 
for PE manifolds.
	\end{rem}
	On the other hand, when $(M,g)$ is a compact Einstein manifold, we prove the following theorems.
	\begin{theorem}\label{lpWcpt}
		Let $(M,g)$ be a compact Einstein manifold with positive scalar curvature. If there exists a $p\in \sq{\frac{n}{2},\infty}$ such that
		\begin{equation}\label{Lp weyl cpt}
			\abs{\abs{\weyl}}_{L^p}\leq\frac{S}{2}\frac{n+1}{n(n-1)}\pa{\frac{4(n-1)}{n(n-2)}+1}^{-\frac{n}{2p}}s(n)^{-1}\mathrm{Vol}(M)^{\frac{1}{p}},
		\end{equation}
		then $(M,g)$ is stable. If the inequality is strict, then $(M,g)$ is strictly stable.
	\end{theorem}
	Moreover, when $(M,g)$ is a compact K\"{a}hler-Einstein manifold, we provide a similar criterion involving the Hermitian analogous of the Weyl tensor, the so-called Bochner tensor: 
	\begin{theorem}\label{lpBcpt}
		Let $(M,g)$ be a compact K\"{a}hler-Einstein manifold with positive scalar curvature. If there exists a $p\in \sq{\frac{n}{2},\infty}$ such that
		\begin{equation}\label{Lp boc cpt}
			\abs{\abs{\boc}}_{L^p}\leq\frac{S}{2}\frac{n-2}{n(n+2)}\pa{\frac{4(n-1)}{n(n-2)}+1}^{-\frac{n}{2p}}s(n)^{-1}\mathrm{Vol}(M)^{\frac{1}{p}},
		\end{equation}
		then $(M,g)$ is stable. If the inequality is strict, then $(M,g)$ is strictly stable.
	\end{theorem}
	We point out that Theorem \ref{lpWcpt} and Theorem \ref{lpBcpt} generalize the stability criteria for the $L^\infty$ and $L^{\frac{n}{2}}$ norms of the Weyl and Bochner tensor provided by the second author in \cite{Kroncke15}. Moreover, the constant $s(n)$ improves the already known results.

	Finally, the last part of the paper is dedicated to compact K\"{a}hler-Einstein and Sasaki-Einsten manifolds in order to provide classification results for the Bochner and contact Bochner tensors, which are the Hermitian and Sasakian counterparts of the Weyl tensor. It was proved by Itoh and Kobayashi that for K\"{a}hler-Einstein manifold of dimension $n=2m$ of positive scalar curvature either the scalar curvature is bounded above by the Bochner tensor or $(M,g)$ is biholomorphically homothetic to $\mathbb{CP}^m$ endowed with the Fubini-Study metric. They also provided an analogous result for Sasaki $\eta$-Einstein manifolds: indeed, either the contact Bochner tensor is greater than or equal to a constant depending on the scalar curvature or $(M,g)$ is $D$-homothetic to a quotient of the standard sphere. Note that some other isolation results were provided for the Bochner tensor in \cite{ChongDongLinRen}, improving the constants obtained by Itoh and Kobayashi. Using the same technique of \cite{BCDM}, we provide the optimal pinching result:
	\begin{theorem}\label{thm bochner intro}
		Let $(M,g,J)$ be a compact K\"{a}hler-Einstein manifold of dimension $n=2m\geq4$ with positive Yamabe invariant. Then either $(M,g,J)$ is biholomorphically homothetic to the complex projective space $\mathbb{CP}^m$ endowed with the Fubini-Study metric or, if $\operatorname{B}\neq0$,
		\begin{align}\label{ineq yamabe and bochner}
			Y(M,[g])\leq n\pa{\int_M\abs{\mathcal{B}}^{\frac{n}{2}}\abs{\operatorname{B}}^{-n}dV_g}^{\frac{2}{n}},
		\end{align}
		where $\mathcal{B}$ is a cubic term in Bochner (see equation \eqref{mathcalboc}).
		Moreover, equality holds if and only if $(M,g,J)$ is locally symmetric.
	\end{theorem}
	Note that Theorem \ref{thm bochner intro} is the analogous of the isolation result for the Weyl tensor established by Theorem 1.1 in \cite{BCDM}.\\
	Our last result provides an improvement of the constants found by Itoh and Kobayashi in \cite{ItohKobayashi} on Sasaki $\eta$-Einstein manifolds. Namely,
	\begin{theorem}\label{thm contact Bochner}
		Let $(M,\phi,\xi,\eta,g)$ be a compact Sasaki $\eta$-Einstein manifold of dimension $n=2m+1\geq5$, with scalar curvature $S>-(n-1)$. Then either $M$ is $D$-homothetic to a quotient of the standard $n$-sphere or, 
		if $n=5$ and $\weyl\neq0$,
		\begin{align*}
			||\operatorname{B}||_{L^{\frac{5}{2}}}>\frac{5\sqrt{10}}{128}(S+4)\mathrm{Vol}(M)^{\frac{2}{5}};
		\end{align*}
		if $n>5$ and $\weyl\neq0$,
		\begin{align*}
			||\operatorname{B}||_{L^{\frac{n}{2}}}\geq \pa{2+\frac{1}{2}\frac{n(n-1)-4}{\sqrt{n(n-2)(n-1)(n+1)}}}^{-1}\frac{(S+n-1)}{n+1}\mathrm{Vol}(M)^{\frac{2}{n}}.
		\end{align*}
	\end{theorem}
	The proof is obtained combining Corollary 1.2 of \cite{BCDM} and the isolation result provided by Itoh and Kobayashi.\\
	The paper is organized as follows: In Section \ref{Sec Preliminaries} we fix the notation and we introduce some useful definitions. Section \ref{Sec ALE} is dedicated to stability and isolation results for ALE manifolds while Section \ref{Sec PE} is dedicated to the proof of Theorem \ref{t-stability PE}. Section \ref{Sec stability cpt} is devoted to the proof of Theorem \ref{lpWcpt} and Theorem \ref{lpBcpt}. The last section is devoted to pinching results for K\"{a}hler-Einstein and Sasaki $\eta$-Einstein manifolds.

	\section*{Acknowledgments} L.B. thanks KTH for the hospitality, Professor P. Mastrolia and Professor G. Catino for their interest and support in this work and D. Dameno and M. Gatti for the useful discussions. L.B. is a member of the Gruppo Nazionale per le Strutture Algebriche, Geometriche e loro Applicazioni (GNSAGA) of INdAM (Istituto Nazionale di Alta Matematica). Furthermore, both authors want to thank the referee for careful reading and many helpful comments which
helped to improve the paper.
	\section{Preliminaries}\label{Sec Preliminaries}
	Let us first fix some notation and conventions. Let $(M,g)$ be a Riemannian manifold of dimension $m\geq3$, given a fixed metric we point out that our convention for the squared norm of a $(p,q)$-tensor field $T$
	is
	\[
	\abs{T}^2=T_{i_1...i_q}^{j_1...j_p}T_{i_1...i_q}^{j_1...j_p}.
	\] 
	Moreover, we define the Laplace-Beltrami operator as $\Delta=\mathrm{tr}\nabla^2$. We define the Riemann curvature tensor in global notation ($(1,3)$-version) as:
	\begin{align*}
		R(X,Y)Z=\nabla_X \nabla_Y Z- \nabla_Y\nabla_X Z- \nabla_{[X,Y]}Z, && \forall\, X,Y,Z \in \mathfrak{X}(M),
	\end{align*}
	where $\mathfrak{X}(M)$ are the smooth vector fields on $M$; while we define its $(0,4)$-version as:
	$$\mathrm{Riem}(X,Y,Z,W)=g(R(X,Y)W,Z)\:, \qquad \forall\: X,\:Y,\:Z,\:W \in \mathfrak{X}(M).$$
	It is well-known that on an $n$-dimensional Riemannian manifold $(M,g)$, $n\geq3$,
	the Riemann curvature tensor $\operatorname{Riem}$ admits the
	decomposition:
	\[
	\operatorname{Riem}=\weyl+\dfrac{1}{n-2}
	\ricc\KN g
	-\dfrac{S}{2(n-1)(n-2)}g\KN g,
	\]
	where $\weyl$ is the \emph{Weyl} tensor,
	$\ricc$ is the \emph{Ricci} tensor and
	$S$ is the \emph{scalar curvature}. Here, $\KN$ denotes the
	\emph{Kulkarni-Nomizu} product: let $h,k$ be two $(0,2)$ tensors, then, in coordinates
	$$(h\KN k)_{ijtl}=h_{it}k_{jl}-h_{il}k_{jt}+h_{jl}k_{it}-h_{jt}k_{jl}.$$
	In particular, with respect to a local
	orthonormal frame, we have
	\begin{align*} 
		R_{ijkt}=W_{ijkt}+\dfrac{1}{n-2}\pa{R_{ik}\delta_{jt}
			-R_{it}\delta_{jk}+R_{jt}\delta_{ik}-R_{jk}\delta_{it}}
		-\dfrac{S}{(n-1)(n-2)}\pa{\delta_{ik}\delta_{jt}-\delta_{it}\delta_{jk}},
	\end{align*}
	where $R_{ij}=R_{ikjk}$ and $S=R_{ii}$. For a Riemannian manifold $(M,g)$ of dimension $3$ the Weyl tensor is
	identically zero while for $n\geq 4$, the Weyl
	tensor is the  traceless part of $\riem$. The $(1,3)$-version of Weyl remains unchanged under conformal
	deformations of the metric $g$: we recall that a \emph{conformal
		deformation} of $g$ is a new metric $\tilde{g}$ obtained by
	rescaling $g$ \emph{via} a smooth positive function $f$, i.e.
	\begin{equation} \label{confchange}
		\tilde{g}=f^2g
	\end{equation}
	and we denote 
	\[
	[g]=\set{\tilde{g}\in \mathcal{M}\,:\,\exists f\in C^{\infty}(M), f>0\,:\,\tilde{g}=f^2g}.
	\]
	A Riemannian manifold is said to be $\emph{Einstein}$ if 
	the Ricci tensor is proportional to the metric, namely
	\begin{align*}
		\ric=\lambda g,\quad\lambda\in\erre,
	\end{align*}
	where $\lambda=S/n$.
	In this case the curvature is encoded by the Weyl tensor and by the value of the scalar curvature; indeed
	\begin{align}\label{riem on Einstein}
		\riem=\weyl+\frac{S}{2n(n-1)}g\KN g.
	\end{align}
	In particular, we say that $(M,g)$ is \emph{Ricci-flat} if the Ricci tensor vanishes identically.\\
	For every $n\geq 4$, the Weyl tensor of an Einstein manifold satisfies
	a fundamental Bochner-Weitzenb\"{o}ck formula
	\[
	\dfrac{1}{2}\Delta\abs{\weyl}^2=
	\abs{\nabla\weyl}^2+\dfrac{2S}{n}\abs{\weyl}^2-
	2Q,
	\]
	where $Q=
	2W_{pqrs}W_{ptru}W_{qtsu}+\frac{1}{2}
	W_{pqrs}W_{pqtu}W_{rstu}$.
	When $n=4$ the Weyl tensor can be seen as a self-adjoint operator from the bundle of $2$-forms to itself
	\begin{align*}
		\mathcal{W}:\Lambda^2 \rightarrow \Lambda^2.
	\end{align*}
	The splitting of the bundle of $2$-forms  ($\Lambda^2=\Lambda_+\oplus\Lambda_-$)
	leads to the decomposition of
	the Weyl operator into a self-dual and an anti-self-dual
	part and a decomposition of the Weyl tensor: 
	\begin{align*}
		\weyl=\weyl^++\weyl^-.
	\end{align*}
	Therefore, when $n=4$ , the Weitzenb\"{o}ck
	formula holds for $W^{\pm}$ and the term $Q^{\pm}$ can be rewritten 
	in a simpler way: 
	\begin{align*}
		Q^{\pm}:=2W_{pqrs}^{\pm}W_{ptru}^{\pm}W_{qtsu}^{\pm}+
		\dfrac{1}{2}W_{pqrs}^{\pm}W_{pqtu}^{\pm}W_{rstu}^{\pm}=
		36\operatorname{det}_{\Lambda_{\pm}}\mathcal{W}_{\pm}.
	\end{align*} 
	In general, for $n\geq4$ we can find estimates for $Q$ in terms of
	$\abs{\weyl}^3$: 
	\begin{align}\label{estimate on Q}
		Q\leq C(n)\abs{\weyl}^3,
	\end{align}
	with $C(4)=\frac{\sqrt{6}}{4}$, $C(5)=\frac{4}{\sqrt{10}}$, $C(6)=\frac{\sqrt{70}}{2\sqrt{3}}$ and $C(n)=2+\frac{1}{2}\frac{n(n-1)-4}{\sqrt{n(n-2)(n-1)(n+1)}}$ for
	$n\geq 6$ ,
	but we cannot simplify the Bochner-Weitzenb\"{o}ck formula for $n>4$ (we point out that the constant for $n=4$ was obtained in \cite[Lemma 3.5]{Huisken}, for $n=5$ was obtained by Tran in \cite{Tran}, for $n=6$ it was obtained by Huisken in \cite{Huisken} adapting an argument due to Tachibana \cite{tachibana2} while the other constants are obtained using Cauchy-Schwartz inequality combined with \cite[Lemma 2.4]{Huisken})\\
	
	We now consider the self-adjoint elliptic operator $\Delta_E:\Gamma(S^2 M)\ra\Gamma(S^2M)$
	\begin{align}\label{def of Delta E}
		\Delta_E h=-\Delta h-2\mathring{R}h,
	\end{align}
	where $(\mathring{R}h)_{ij}=R_{ipjq}h_{pq}$. The operator
	$\Delta_E$ is known as the \emph{Einstein operator}.
	\begin{rem}
		Throughout the paper we will denote 
		\begin{align*}
			(\mathring{W}h)_{ij}=W_{ipjq}h_{pq};\quad(\mathring{B}h)_{ij}=B_{ipjq}h_{pq},
		\end{align*}
		where $\weyl$ and $\boc$ denote the Weyl and Bochner tensor of $g$.
	\end{rem}
	\begin{rem}
			The Einstein operator satisfies the following Bochner formulas for every $h\in\Gamma(S^2 M)$ tensor with compact support
			\begin{align}\label{bochner e oper D1}
				(\Delta_E h,h)_{L^2}&=||D_1h||^2_{L^2}+2\frac{S}{n}\abs{\abs{h}}_{L^2}^2-4(\mathring{R}h,h)_{L^2}-2\abs{\abs{\diver h}}_{L^2}^2,
			\end{align}
			and
			\begin{align}\label{bochner e oper D2}
				(\Delta_E h,h)_{L^2}&=||D_2h||^2_{L^2}-\frac{S}{n}||h||_{L^2}-(\mathring{R}h,h)_{L^2}+\abs{\abs{\diver h}}_{L^2}^2,
			\end{align}
			where $D_1$ and $D_2$ are two differential operators defined by
			\begin{align*}
				D_1h(X,Y,Z)&=\frac{1}{\sqrt{3}}\pa{\nabla_Xh(Y,Z)+\nabla_Yh(Z,X)+\nabla_Zh(X,Y)},\\
				D_2h(X,Y,Z)&=\frac{1}{\sqrt{2}}\pa{\nabla_Xh(Y,Z)-\nabla_Yh(Z,X)},
			\end{align*}
			for any $X,Y,Z\in\mathfrak{X}(M)$.
			See \cite{Besse, Koiso78} for more details.
			These formulas will be used in the proof of Theorems \ref{t-stability PE}, \ref{lpWcpt} and \ref{lpBcpt}. 
	\end{rem} 
	In this paper, compact manifolds are always assumed to have an empty boundary, unless stated otherwise.
	\begin{defi} We say that a compact Einstein manifold $(M,g)$ is \emph{stable} (\emph{strictly stable}) if the Einstein operator is non-negative (\emph{positive}) on transverse traceless tensors (namely, symmetric $(0,2)$-tensors that are trace free and whose divergence is identically zero, briefly referred to as TT tensors). We call $(M,g)$ \emph{unstable} if the Einstein operator admits negative eigenvalues on transverse traceless tensors. 
	\end{defi}
	More in general, for a non-compact Einstein manifold we say that it is strictly stable if there exists a constant $C>0$ such that 
	\begin{equation}\label{def stab}
		(\Delta_Eh,h)_{L^2}\geq C\abs{\abs{h}}_{L^2}^2
	\end{equation}
	for all compactly supported $TT$ tensors; we call $(M,g)$ stable if \eqref{def stab} holds for $C=0$.
	In particular, we will be interested in stability results on ALE manifolds and on \emph{asymptotically hyperbolic} (AH) manifolds that are Einstein.
	
	\subsection{ALE manifolds}
	In the first part of the paper we will focus on specific non-compact Ricci-flat Riemannian manifolds, namely \emph{asymptotically locally Euclidean} manifolds. 
	\begin{defi}[ALE manifold]\label{d-ale} An $n$-dimensional complete Riemannian manifold $(M,g)$ is said to be \emph{asymptotically locally Euclidean} (ALE) of order $\tau>0$ if there exists a compact set $K\subset M$ such that $M\setminus K$ has coordinates at infinity: namely there exists a diffeomorphism $\phi:M\setminus K\ra\pa{\erre^n\setminus\ol{B}_1}/\Gamma$, where $\Gamma$  is a finite subgroup of $SO(n)$ acting freely on $\erre^n$, such that
		\begin{align*}
			\abs{\pa{\nabla^{eucl}}^k\pa{\phi_*g-g_{eucl}}}_{eucl}=O(r^{-\tau-k})
		\end{align*}
		holds on $\pa{\erre^n\setminus\ol{B}_1}/\Gamma$. 
	\end{defi}
	Many of these spaces arise from a physical context as solutions to the Einstein equation. The examples found by physicists admit an \emph{hyperk\"{a}hler} structure, that is
	\begin{defi}[hyperk\"{a}hler structure]
		An almost complex structure on $M$ is an endomorphism $J:TM\ra TM$ satisfying
		\[
		J^2=-I.
		\]
		A \emph{hyperk\"{a}hler structure} on a Riemannian manifolds $(M,g)$ is given by a set of three integrable almost complex structures (I,J,K), such that
		\begin{align*}
			&IJ=-JI=K;\\
			&\nabla I=\nabla J=\nabla K=0;\\
			&g(IX,IY)=g(JX,JY)=g(KX,KY)=g(X,Y), \quad\forall X,Y\in\mathfrak{X}(M).
		\end{align*}
	\end{defi}
	A well-know example of Ricci-flat ALE was found by Eguchi and Hanson in \cite{EguchiHanson} and it is known as Eguchi-Hanson space. This space is an ALE complete $4$-dimensional non-compact Riemannian manifold satisfying the Einstein equation with $\Gamma=\mathbb{Z}_2$ and it is hyperk\"{a}hler (thus Ricci-flat). More in general, Hitchin conjectured the existence of other ALE complete $4$-dimensional non-compact Riemannian manifolds satisfying the Einstein equation for other finite subgroups $\Gamma\subset SU(2)$. The conjecture was proved true by Kronheimer (see \cite{Kronheimer}), who showed that every $4$-dimensional hyperk\"{a}hler ALE manifold is diffeomorphic to a minimal resolution of $\erre^4\setminus\set{0}/\Gamma$, where $\Gamma$ is a discrete subgroup of $SU(2)$ acting freely on $\mathbb{S}^3$. It is an open problem whether other
	examples of Ricci-flat ALE manifolds exist in dimension four.\\
	One of the main tools we are going to use to prove Theorem \ref{isolation thm intro} is the following Sobolev inequality, which holds for complete non-compact Riemannian manifolds of dimension $n\geq2$ with non-negative Ricci and having Euclidean volume growth.
	\begin{theorem}[Theorem 1.1 of \cite{SharpSobolev}]\label{sob thm}
		Let $(M,g)$ be a non-compact, complete $n$-dimensional manifold ($n\geq2$) with $\ric\geq0$ and having Euclidean volume growth. If $q\in[1,n)$, then
		\begin{align}\label{sobolev}
			\pa{\int_M\abs{f}^{q^*}dV_g}^{\frac{1}{q^*}}\leq \mathrm{AT}(n,q)\mathrm{AVR}^{-\frac{1}{n}}_g\pa{\int_M\abs{\nabla f}^qdV_g}^{\frac{1}{q}},\quad\forall f\in C^\infty_0,
		\end{align}
		where $q^*=\frac{qn}{n-q}$, 
		\begin{align*}
			\mathrm{AT}(n,q)&=\pi^{-\frac{1}{2}}n^{-\frac{1}{q}}\pa{\frac{q-1}{n-q}}^{\frac{q-1}{q}}\pa{\frac{\Gamma\pa{1+\frac{n}{2}}\Gamma(n)}{\Gamma\pa{\frac{n}{q}}\Gamma\pa{1+\frac{n(q-1)}{q}}}}^{\frac{1}{n}},\\
			\mathrm{AVR}_g&=\lim_{r\ra\infty}\frac{\mathrm{Vol_g(B_r(x))}}{\omega_nr^n}
		\end{align*}
		and $C^\infty_0$ denotes the space of smooth compactly supported functions.
	\end{theorem}
	\subsection{AH manifolds} The second part of the paper is devoted to another class of non-compact manifolds, namely \emph{asymptotically hyperbolic} manifolds.
%
\begin{defi}[AH manifolds]
		 A complete Riemannian manifold $(M,g)$ is called conformally compact of class $C^{k,\alpha}$ if $M$ is the interior of a manifold $\overline{M}$ with boundary $\partial M$ and the metric $h=\rho^2g$ can be extended to a $C^{k,\alpha}$-metric on all of $\overline{M}$. Here, $\rho:\overline{M}\to [0,\infty)$ is a boundary defining function for $\overline{M}$, that is $\rho^{-1}(0)=\partial M$ and $d\rho$ is nowhere vanishing on $\partial M$. 
Note that the metric $h|_{\partial M}$ on $\partial M$ depends on $\rho$ but its conformal class is independent of it. We call $(\partial M,[h|_{\partial N}])$ the {conformal boundary} of $(M,g)$. Furthermore, we call $(M,g)$ {asymptotically hyperbolic} (AH), if $|d\rho|_h\equiv 1$ on $\partial M$.
\end{defi}	
	
	A standard computation, which makes use of the conformal transformation of the curvature (see e.g. \cite{Besse}) shows that when $(M,g)$ is asymptotically hyperbolic, then its sectional curvature converges to $-1$,
	\begin{align*}
		K=-1+O(\rho).
	\end{align*}
	Moreover, an Einstein manifold that is AH is called Poincaré-Einstein, or simply PE. In this case the Ricci tensor and the scalar curvature curvature satisfy
	\begin{align*}
		\ric=-(n-1)g
	\end{align*}
	and 
	\begin{align*}
		S=-n(n-1),
	\end{align*}
	respectively. Observe that, since the Weyl functional is invariant under conformal change of the metric, on an AH manifold we have that 
	\begin{align*}
		||\weyl_g||_{L^{\frac{n}{2}}(g)}=||\weyl_h||_{L^{\frac{n}{2}}(h)}
	\end{align*}
	where $\weyl_g$ is the Weyl and $\weyl_h$ are the Weyl tensor with respect to the metric $g$ and $h$, respectively. In particular,
	since $(\ol{M},h)$ is a compact manifold with boundary, $||\weyl_h||_{L^{\frac{n}{2}}(h)}$ is finite and then we deduce that also $||\weyl||_{L^{\frac{n}{2}}(g)}$ does not diverge.
	\subsection{Compact K\"{a}hler and Sasaki manifolds}
	The second part of the paper is dedicated to compact K\"{a}hler-Einstein and Sasaki-Einstein manifolds, therefore we recall here some useful definitions.
	\begin{defi}
		Let $(M,g)$ be a Riemannian manifold of dimension $n=2m$, an almost complex structure on $M$ is an endomorphism $J:TM\ra TM$ satisfying
		\[
		J^2=-I.
		\]
		When $J$ is integrable and the metric $g$ is Hermitian (namely, $g(JX,JY)=g(X,Y)$), the triple $(M,g,J)$ is said to be K\"{a}hler.
	\end{defi}
	Note that, when $(M,g,J)$ is Kähler, then the Riemann tensor and the Ricci tensor satisfy the following identities
	\begin{align*}
		\mathrm{Riem}(X,Y,Z,W)&=\mathrm{Riem}(JX,JY,Z,W)=\mathrm{Riem}(X,Y,JZ,JW),\\
		\mathrm{Ric}(X,Y)&=\mathrm{Ric}(JX,JY),
	\end{align*}
	for each $X,Y,Z,W \in \mathfrak{X}(M)$.\\
	In \cite{Bochner}, Bochner introduced a complex counterpart of the Weyl tensor, known as \emph{Bochner tensor}. 
	\begin{defi}[Bochner tensor]
		Let $(M, g, J)$ be a K\"{a}hler manifold of dimension $n=2m\geq4$, then the \emph{Bochner curvature tensor} is the $(0,4)$ Hermitian tensor		
		whose components, with respect to complex local coordinates, are
		\begin{align*}
			B_{i\ol{j}k\ol{l}}=&R_{i\ol{j}k\ol{l}}-\frac{1}{m+2}\pa{R_{i\ol{j}}g_{k\ol{l}}+R_{k\ol{j}}g_{i\ol{l}}+R_{k\ol{l}}g_{i\ol{j}}+R_{i\ol{l}}g_{k\ol{j}}}+\frac{R}{(m+1)(m+2)}\pa{g_{i\ol{j}}g_{k\ol{l}}+g_{k\ol{j}}g_{i\ol{l}}},
		\end{align*} 
		where $R_{i\ol{j}k\ol{l}}$, $R_{i\ol{j}}$ are the components of the Riemann curvature tensor and of the Ricci tensor, respectively, and $R=g^{i\ol{j}}R_{i\ol{j}}$ is the complex scalar curvature. Note that $R=\frac{S}{2}$, where $S$ is the standard scalar curvature.
	\end{defi}
	In \cite{Tachibana}, Tachibana managed to express the Bochner tensor in real coordinates. In particular, its components, with respect to a local orthonormal frame, are
	\begin{align*}
		B_{ijst}=&R_{ijst}-\frac{1}{n+4}[R_{is}\delta_{jt}+R_{jt}\delta_{is}-R_{it}\delta_{js}-R_{js}\delta_{it}\\
		&+J^r_iR_{rs}J^t_j+J^r_jR_{rt}J_i^s-J^r_iR_{rt}J^s_j-J^r_jR_{rs}J^t_i+2J^r_iR_{rj}J_s^t+2J^j_iJ^r_sR_{rt}]\\
		&+\frac{S}{(n+2)(n+4)}[\delta_{is}\delta_{jt}-\delta_{it}\delta_{js}+J^s_iJ^t_j-J^t_iJ^s_j+2J^j_iJ^t_s].
	\end{align*}
	We remark that, using real coordinates, the Bochner tensor satisfies the following identities, that can be obtained by a straightforward computation:
	\begin{align*}
		&B_{ijst}=-B_{jist}=B_{jits};\\
		&B_{ijst}=B_{stij};\\
		&B_{ijst}+B_{istj}+B_{itjs}=0;\\
		&B_{pjpt}=0;\\
		&B_{ijst}=J^p_iJ^q_jB_{pqst},
	\end{align*}
	and the last identity means that $\mathrm{B}$ is $J$-invariant, that is
	\begin{align*}
		B(JX,JY,Z,W)=B(X,Y,Z,W)\quad\forall X,Y,Z,W\in\mathfrak{X}(M).
	\end{align*}
	When $(M^{n},J,g)$, $n=2m$, is K\"{a}hler-Einstein, the Riemann curvature tensor rewrites just in terms of the Bochner tensor and of the value of the complex scalar curvature: 
	\begin{align*}
		R_{i\ol{j}k\ol{l}}=&B_{i\ol{j}k\ol{l}}+\frac{R}{m(m+1)}\pa{g_{i\ol{j}}g_{k\ol{l}}+g_{k\ol{j}}g_{i\ol{l}}}
	\end{align*}
	or, equivalently, in real coordinates:	
	\begin{align*}
		R_{ijst}=B_{ijst}+\frac{S}{n(n+2)}\pa{\delta_{is}\delta_{jt}-\delta_{it}\delta_{js}+J^s_iJ^t_j-J^t_iJ^s_j+2J^j_iJ^t_s}.
	\end{align*} 
	Moreover, it satisfies the following Bochner-Weitzenb\"{o}ck formula (see \cite{ChongDongLinRen} for a proof)
	\begin{align}\label{bw for b}
		\frac{1}{2}\Delta\abs{\boc}^2&=\abs{\nabla \boc}^2-8B_{\ol{j}pq\ol{l}}B_{\ol{p}ik\ol{q}}B_{\ol{i}jl\ol{k}}+16B_{i\ol{l}\ol{p}q}B_{p\ol{q}\ol{j}k}B_{j\ol{k}\ol{i}l}+\frac{2}{m}R\abs{B}^2\\
		&=:\abs{\nabla \boc}^2-2\mathcal{B}+\frac{2}{m}R\abs{B}^2,\notag
	\end{align}
	where
	\begin{align}\label{mathcalboc}
		\mathcal{B}:=4B_{\ol{j}pq\ol{l}}B_{\ol{p}ik\ol{q}}B_{\ol{i}jl\ol{k}}-8B_{i\ol{l}\ol{p}q}B_{p\ol{q}\ol{j}k}B_{j\ol{k}\ol{i}l}
	\end{align}
	is a cubic term in Bochner. Introducing two $m^2\times m^2$ Hermitian matrices $H$ and $K$ as
	\begin{align*}
		H=\pa{B_{\ol{k}ij\ol{l}}},\quad K=\pa{B_{i\ol{j}\ol{k}l}}
	\end{align*}
	we have that \eqref{bw for b} can be rewritten as
	\begin{align*}
		\frac{1}{2}\Delta\abs{\boc}^2=\abs{\nabla \boc}^2-8\mathrm{tr}(H^3)+16\mathrm{tr}(K^3)+\frac{2}{m}R\abs{B}^2
	\end{align*}
	and Lemma 2.4 of \cite{Huisken} (for more details see also Section 4 of \cite{ChongDongLinRen}; see e.g. equations (4.4) and (4.5)), implies
	\begin{align}\label{estimate mathcal B}
		\abs{\mathcal{B}}\leq\frac{3(m^2-2)}{2m\sqrt{m^2-1}}\abs{\boc}^3.
	\end{align}
	In fact, we have that $\mathrm{tr}(H)=0=\mathrm{tr}(K)$ and since $\abs{\boc}^2=4B_{i\ol{j}k\ol{l}}\ol{B_{i\ol{j}k\ol{l}}}$, 
	\begin{align*}
		\mathrm{tr}(H^3)\leq\frac{m^2-2}{8m\sqrt{m^2-1}}\abs{\boc}^3,\\
		\mathrm{tr}(K^3)\leq\frac{m^2-2}{8m\sqrt{m^2-1}}\abs{\boc}^3.
	\end{align*}
	and hence \eqref{estimate mathcal B}.
	Similarly to the K\"{a}hler case, it is possible to define an analogue of the Weyl tensor when considering a Sasakian manifold; this counterpart of the Weyl tensor is known as \emph{contact Bochner tensor}. 
	\begin{defi}[Contact Bochner tensor]
		Let $(M,\phi,\xi,\eta,g)$ be a Sasaki manifold of dimension $n=2m+1\geq 5$, then the \emph{contact Bochner tensor} is the $(0,4)$-tensor whose components, with respect to a local orthonormal frame, are
		\begin{align*}
			B_{ijst}=&R_{ijst}-\frac{1}{n+3}[R_{is}\delta_{jt}+R_{jt}\delta_{is}-R_{it}\delta_{js}-R_{js}\delta_{it}\\
			&+R_{ir}\phi^r_s\phi_{jt}+R_{jr}\phi^r_t\phi_{is}-R_{ir}\phi^r_t\phi_{js}-R_{jr}\phi^r_s\phi_{it}+2R_{ir}\phi^r_j\phi_{st}+2\phi_{ij}R_{sr}\phi^r_t\notag\\
			&-R_{is}\eta_j\eta_t-R_{jt}\eta_i\eta_s+R_{it}\eta_j\eta_s+R_{js}\eta_i\eta_t]+\frac{k+n-1}{n+3}[\phi_{is}\phi_{jt}-\phi_{it}\phi_{js}+2\phi_{ij}\phi_{st}]\notag\\
			&\frac{k-4}{n+3}[\delta_{is}\delta_{jt}-\delta_{it}\delta_{js}]-\frac{k}{n+3}[\delta_{is}\eta_j\eta_t+\delta_{jt}\eta_i\eta_s-\delta_{it}\eta_j\eta_s-\delta_{js}\eta_i\eta_t],\notag
		\end{align*}
		where $k=\frac{S+n-1}{n+1}$. 
	\end{defi}
	Analogously to the Bochner tensor, the contact Bochner tensor satisfies some symmetries, more precisely:
	\begin{align*}
		&B_{ijst}=-B_{jist}=B_{jits};\\
		&B_{ijst}+B_{istj}+B_{itjs}=0;\\
		&B_{ijst}=B_{stij};\\
		&\xi^pB_{pjst}=0;\\
		&B_{ijst}=\phi^p_i\phi^q_jB_{pqst};\\
		&B_{pjpt}=0.
	\end{align*}
	\section{Isolation and stability results for the Weyl functional on Ricci-flat ALE manifolds}\label{Sec ALE}
	The aim of this section is to first provide an isolation result for the Weyl functional on Ricci-flat ALE manifolds of dimension $n$. 
	In the second part of the section we will compare this result to a stability one, finding a small gap in which the manifold is stable but not flat.
	\begin{proof}[Proof of Theorem 1.1]
	Assume that $\weyl\not\equiv 0$.
	We recall that on any Einstein manifold the Weyl tensor satisfies the following Bochner-Weitzenb\"{o}ck formula		
	\begin{align}\label{BW}
		\dfrac{1}{2}\Delta\abs{\weyl}^2=
		\abs{\nabla\weyl}^2+\dfrac{2S}{n}\abs{\weyl}^2-
		2Q,
	\end{align}
	where $Q=
	2W_{pqrs}W_{ptru}W_{qtsu}+\frac{1}{2}
	W_{pqrs}W_{pqtu}W_{rstu}$.
	Choose $q\in( \frac{1}{2},1]$and set $f:=\abs{\weyl}^{2q}$.
	Outside the zero set of $\abs{\weyl}$, we have 
	\begin{align*}
		\Delta f&=4q(q-1)f^{1-\frac{1}{q}}\abs{\nabla\abs{\weyl}}^2+qf^{1-\frac{1}{q}}\Delta\abs{\weyl}^2
		\\			&
		=4q(q-1)f^{1-\frac{1}{q}}\abs{\nabla\abs{\weyl}}^2+qf^{1-\frac{1}{q}}\pa{2\abs{\nabla\weyl}^2-4Q}.
	\end{align*}
	Making use of the refined Kato inequality for Einstein manifolds (see \cite{Bando} for a proof), which is
	\[\abs{\nabla\abs{\weyl}}\leq\sqrt{\frac{n-1}{n+1}}\abs{\nabla\weyl},\]
	we get
	\begin{align*}
		\Delta f&\geq2q\pa{2(q-1)+\frac{n+1}{n-1}}f^{1-\frac{1}{q}}\abs{\nabla\abs{\weyl}}^2-4qf^{1-\frac{1}{q}}Q.
	\end{align*}
	Thus, integrating over a compact set $K\subset M$ with smooth boundary satisfying the condition $K\cap (\abs{\weyl}^{-1}(0))=\emptyset$, we get
	\begin{align*}
		\int_{K}\Delta f\; dV_g&\geq\int_{K}\pa{2q\pa{2(q-1)+\frac{n+1}{n-1}}f^{1-\frac{1}{q}}\abs{\nabla\abs{\weyl}}^2-4qf^{1-\frac{1}{q}}Q}dV_g,
	\end{align*}
	and by the divergence theorem we obtain
	\begin{align}\label{divergence_inequality}
		\int_{\partial K} g(\nabla f,\nu)\; d\sigma_g+
		4q\int_{ K} f^{1-\frac{1}{q}}Q\; dV_g
		&\geq 2q\pa{2(q-1)+\frac{n+1}{n-1}}\int_{K}f^{1-\frac{1}{q}}\abs{\nabla\abs{\weyl}}^2dV_g,
	\end{align}
	where $\nu$ is the outward unit normal. 
	
	Now we want to show that the left-hand side admits a uniform bound for all  compact $K\subset M$ with $K\cap (\abs{\weyl}^{-1}(0))=\emptyset$. 
	Observe first that the integrands on the left-hand side are bounded, since
	\begin{align*}
		|g(\nabla f,\nu)|\leq |\nabla f|\leq 2q |\weyl|^{2q-1}|\nabla \weyl|\leq C,\qquad
		|f^{1-\frac{1}{q}}Q|\leq C|\weyl|^{2q+1}\leq C.
	\end{align*}
	Here, we  also used that $q>\frac{1}{2}$. In order to study the decay of the integrands at infinity, let us recall that every $n$-dimensional Ricci-flat ALE manifold is ALE of order $n$ (see \cite{KronckeSzabo}) and therefore, $|\weyl|=O(r^{-n-2})$ and $|\nabla\weyl|=O(r^{-n-3})$ as $r\to\infty$. This implies that 
	\begin{align*}
		|g(\nabla f,\nu)|\leq C r^{-(n+2)(2q-1)-n-3}\leq Cr^{-n-3} ,\qquad |f^{1-\frac{1}{q}}Q|\leq  C r^{-(n+2)(2q+1)}\leq C r^{-(n+2)}.
	\end{align*}
	Using all these estimates, we can bound both integrands (which are defined on $M\setminus (\abs{\weyl}^{-1}(0))$) by a bounded and continuous function $F:M\to \erre$ which vanishes on $\abs{\weyl}^{-1}(0)$ and satisfies $|F|=O(r^{-(n+2)})$ as $r\to\infty$.
	We recall that Einstein metrics are real-analytic in suitable coordinates (see e.g.\ \cite[Theorem 5.26]{Besse}). Therefore, by a result of DeTurck and Kazdan (see e.g.\ \cite[Theorem 5.26]{Besse}), $M$ has an atlas with real	analytic transition functions and in the corresponding charts, $|W|^2$ is a real analytic function. 
	We can conclude that the zero set of the function $\abs{\weyl}$ has zero measure if $\weyl$ is not identically zero (see e.g.\ \cite[p.\ 240]{Federer}). Moreover, $\abs{\weyl}$ cannot vanish on an open subset of $M$ unless $(M,g)$ is flat.\\
	Now take an exhaustion of $M\setminus (\abs{\weyl}^{-1}(0))$ by compact subsets \[K_\eps:=\set{x\in M\,:\, \abs{\weyl}(x)\geq\eps>0}\]
	(observe that these sets are compact since they are closed and bounded, due to the decay of the Weyl tensor at infinity). We point out that, $\abs{\weyl}$ is smooth on $K_\eps$ and due to Sard's Theorem, for almost every $\eps$,
	we have that $\partial K_\eps$ is a smooth hypersurface and the outward unit normal $\nu$ can be expressed as
	\[\nu=-\frac{\nabla\abs{\weyl}}{\abs{\nabla\abs{\weyl}}}.\]
	As a consequence, considering $K_\eps$, \eqref{divergence_inequality} rewrites as
	\begin{align*}
		\int_{\partial K_\eps} -g\pa{\nabla f,\frac{\nabla\abs{\weyl}}{\abs{\nabla\abs{\weyl}}}}\; d\sigma_g+
		4q\int_{ K_\eps} f^{1-\frac{1}{q}}Q\; dV_g
		&\geq 2q\pa{2(q-1)+\frac{n+1}{n-1}}\int_{K_\eps}f^{1-\frac{1}{q}}\abs{\nabla\abs{\weyl}}^2dV_g,
	\end{align*}
	that is
	\begin{align*}
		-\int_{\partial K_\eps} 2q\abs{\weyl}^{2q-1}\abs{\nabla\abs{\weyl}}\; d\sigma_g+
		4q\int_{ K_\eps} f^{1-\frac{1}{q}}Q\; dV_g
		\geq 2q\pa{2(q-1)+\frac{n+1}{n-1}}\int_{K_\eps}f^{1-\frac{1}{q}}\abs{\nabla\abs{\weyl}}^2dV_g
	\end{align*}
	and since
	\[-\int_{\partial K_\eps} 2q\abs{\weyl}^{2q-1}\abs{\nabla\abs{\weyl}}\,d\sigma\leq 0,\]
	we deduce the validity of
	\begin{align}
	\label{inequality_Keps}
		4q\int_{ K_\eps} f^{1-\frac{1}{q}}Q\; dV_g
		\geq 2q\pa{2(q-1)+\frac{n+1}{n-1}}\int_{K_\eps}f^{1-\frac{1}{q}}\abs{\nabla\abs{\weyl}}^2dV_g.
	\end{align}
	By previous estimates, we know that $f^{1-\frac{1}{q}}Q\in L^1(M)$ and by dominated convergence theorem we have
	\begin{align}
	\label{limit LHS}
\lim_{\eps\ra 0}\int_{ K_\eps} f^{1-\frac{1}{q}}Q\; dV_g&=\int_{M\setminus (\abs{\weyl}^{-1}(0))}f^{1-\frac{1}{q}}Q\; dV_g
		=\int_Mf^{1-\frac{1}{q}}Q\; dV_g,
	\end{align}
	where the last equality follows from the fact that the zero set of $\weyl$ has zero measure.
	For every subset $A\subset M$, let $\chi_A$ be the characteristic function of $A$, given by
		\begin{align*}
		\chi_{A}:=\begin{cases}
			1 \quad\text{on }A,\\
			0\quad\text{on }M\setminus A.
		\end{cases}
	\end{align*}
	Define
	\[F_\eps:=f^{1-\frac{1}{q}}\abs{\nabla\abs{\weyl}}^2\chi_{K_\eps}.\]
	Since $f^{1-\frac{1}{q}}\abs{\nabla\abs{\weyl}}^2\geq0$ on $K_\eps$ and due to the fact that for every $\eps_1,\eps_2$ such that $\eps_2\leq\eps_1$, we get $K_{\eps_1}\subseteq K_{\eps_2}$,
	we deduce
	\[F_{\eps_2}=F_{\eps_1}+f^{1-\frac{1}{q}}\abs{\nabla\abs{\weyl}}^2\chi_{K_{\eps_2}\setminus K_{\eps_2}}\geq F_{\eps_1}.\]
	Taking $\eps\ra0$, we get
	$F_{\eps}\ra f^{1-\frac{1}{q}}\abs{\nabla\abs{\weyl}}^2\chi_{M\setminus\pa{\abs{\weyl}^{-1}(0)}}$.
	We highlight that $F_{\eps}\in L^1(M)$ since $F_{\eps}$ has compact support and $\abs{\weyl}$ is nowhere vanishing on $K_{\eps}$. Moreover, for every $\eps>0$, by monotone convergence theorem (and using that $\abs{\weyl}^{-1}(0)$ has zero measure), we deduce 
	\begin{align*}
		\int_{M}f^{1-\frac{1}{q}}\abs{\nabla\abs{\weyl}}^2\,dV_g&=\int_{M}f^{1-\frac{1}{q}}\abs{\nabla\abs{\weyl}}^2\chi_{M\setminus\pa{\abs{\weyl}^{-1}(0)}}\,dV_g\\&=\lim_{\eps\ra 0}\int_{M} F_\eps\,dV_g\\
		&=\lim_{\eps\ra 0}\int_{K_\eps}f^{1-\frac{1}{q}}\abs{\nabla\abs{\weyl}}^2\,dV_g,
	\end{align*}
	which is finite since it is bounded by the left-hand side of \eqref{limit LHS}. As a consequence, we deduce
	\begin{align}\label{integrate_Weyl}
		4q\int_{ M} f^{1-\frac{1}{q}}Q\; dV_g
		&\geq 2q\pa{2(q-1)+\frac{n+1}{n-1}}\int_{M}f^{1-\frac{1}{q}}\abs{\nabla\abs{\weyl}}^2dV_g,
	\end{align}
	which yields
	\begin{align*}
		0&\geq \int_M\sq{2q\pa{2(q-1)+\frac{n+1}{n-1}}\abs{\weyl}^{2q-2}\abs{\nabla\abs{\weyl}}^2-4q\abs{\weyl}^{2q-2}Q}dV_g\\
		&=\int_M\sq{\frac{2}{q}\pa{2(q-1)+\frac{n+1}{n-1}}\abs{\nabla\abs{\weyl}^q}^2-4q\abs{\weyl}^{2q-2}Q}dV_g\\
		&\geq\int_M\sq{\frac{2}{q}\pa{2(q-1)+\frac{n+1}{n-1}}\abs{\nabla\abs{\weyl}^q}^2-4qC(n)\abs{\weyl}^{2q+1}}dV_g.
	\end{align*}
	To proceed, note first that by \cite[Theorem 1.1]{SharpSobolev}, we have for every smooth compactly supported function $v$, that
	\begin{align}\label{sharp sob}
		||\nabla v||_{L^2}^2\geq D(n)^{-2}||v||_{L^p}^2,
	\end{align}
	where $p=\frac{2n}{n-2}$ and
	\begin{align}\label{def of D}
		D(n)&=\pa{\lim_{r\ra+\infty}\frac{\mathrm{Vol_g}(B_r(x))}{\omega_nr^n}}^{-\frac{1}{n}}\sqrt{\frac{1}{(n-2)n\pi}}\pa{\frac{\Gamma(n)}{\Gamma(\frac{n}{2})}}^{\frac{1}{n}}
		\\			&		
		=			\frac{1}{\sqrt{\pi n(n-2)}}\pa{\frac{\Gamma(n)}{\Gamma\pa{\frac{n}{2}}}}^{\frac{1}{n}}\abs{\Gamma}^{\frac{1}{n}}
		\notag
		.
	\end{align} 
	Since $C^{\infty}_{0}\subset H^2_1$ is dense (see e.g.\ \cite[Theorem 2.7]{HebeySob}), the inequality also holds for $v\in H^2_1$.
	For $v=\abs{\weyl}^{q}$, we know by the decay of $\weyl$ that $v\in L^2(M)$ and using \eqref{integrate_Weyl}, we also see that $\nabla v\in L^2(M)$.
	Hence, we have
	\begin{align}\label{sharp sob for v}
		||\nabla \abs{\weyl}^q||_{L^2}^2\geq D(n)^{-2}||\abs{\weyl}^q||_{L^p}^2,
	\end{align}
	where $p=\frac{2n}{n-2}$.	
	Then, using H\"{o}lder inequality and \eqref{sharp sob for v}, we deduce
	\begin{align*}
		0&\geq\int_M\sq{\frac{2}{q}\pa{2(q-1)+\frac{n+1}{n-1}}\abs{\nabla\abs{\weyl}^q}^2}dV_g-4qC(n)||\weyl||_{L^\frac{n}{2}}||\,\abs{\weyl}^q||_{L^p}^2\\
		&\geq\frac{2}{q}\pa{2(q-1)+\frac{n+1}{n-1}}D(n)^{-2}||\,\abs{\weyl}^q||_{L^p}^2-4qC(n)||\weyl||_{L^\frac{n}{2}}||\,\abs{\weyl}^q||_{L^p}^2\\
		&=\sq{\frac{2}{q}\pa{2(q-1)+\frac{n+1}{n-1}}D(n)^{-2}-4qC(n)||\weyl||_{L^\frac{n}{2}}}||\,\abs{\weyl}^q||_{L^p}^2,
	\end{align*}
	which is equivalent to
	\begin{align}\label{rem}
		||\weyl||_{L^\frac{n}{2}}\geq\frac{1}{2q^2}\pa{2(q-1)+\frac{n+1}{n-1}}D(n)^{-2}C(n)^{-1}.
	\end{align}
	For $n\geq6$, substituting $q=2-\frac{n+1}{n-1}>\frac{1}{2}$ yields
	\begin{align*}
		||\weyl||_{L^\frac{n}{2}}\geq\frac{1}{2\pa{2-\frac{n+1}{n-1}}^2}\sq{2\pa{1-\frac{n+1}{n-1}}+\frac{n+1}{n-1}}D(n)^{-2}C(n)^{-1},
	\end{align*}
	which implies the statement for $n\geq 6$.
	For $n=4$ we set $q=\frac{1}{2}+\delta$, which yields 
	\begin{align*}
		||\weyl||_{L^\frac{n}{2}}\geq\frac{1}{2(\frac{1}{2}+\delta)^2}\pa{2\pa{\frac{1}{2}+\delta-1}+\frac{5}{3}}D(4)^{-2}C(4)^{-1}.
	\end{align*}
	Then, passing to the limit for $\delta\ra0$, we deduce the validity of the statement.
\end{proof}
	\begin{rem}
		The value of $q$ for the case $n\geq 6$ in the proof of Theorem \ref{isolation thm intro} is obtained maximizing the function
		\[f(q):=\frac{1}{2q^2}\pa{2(q-1)+\frac{n+1}{n-1}}.\]
		In this way, we obtain the best constant for equation \eqref{rem}. 
	\end{rem}
	Note that the proof of the theorem was mainly based on the use of the Bochner-Weitzenb\"{o}ck formula together with the sharp Sobolev inequality \eqref{sharp sob}. This last inequality can be a fundamental tool to find a stability criterion for ALE manifolds which can be expressed in terms of the Weyl functional and the constant $D(n)$.
	\begin{proof}[Proof of Theorem \ref{stability thm intro}]
		Integrating \eqref{ineq weyl h point} over $M$, for every $h$ compactly supported $TT$ tensor, we get
		\begin{align}\label{ineq weyl h}
			\int_M\pa{W_{ijkl}h_{ik}h_{jl}}dV_g\leq \sqrt{\frac{(n+2)(n-1)}{n(n+1)}}||\weyl||_{L^{\frac{n}{2}}}||h||_{L^p}^2,
		\end{align}
		where we have used H\"{o}lder inequality and $p=\frac{2n}{n-2}$. Exploiting \eqref{ineq weyl h}, it follows
		\begin{align*}
			(\Delta_E h,h)_{L^2}&=(-\Delta h-2\mathring{R}h,h)_{L^2}\\
			&=(-\mathrm{tr}\nabla^2 h-2\mathring{\weyl}h,h)_{L^2}\\
			&=||\nabla h||_{L^2}^2-2\int_MW_{ijkl}h_{ik}h_{jl}dV_g\\
			&\geq||\nabla \abs{h}||_{L^2}^2- 2\sqrt{\frac{(n+2)(n-1)}{n(n+1)}}||\weyl||_{L^{\frac{n}{2}}}||h||_{L^p}^2\\
			&\geq D(n)^{-2}||h||_{L^p}^2-2\sqrt{\frac{(n+2)(n-1)}{n(n+1)}}||\weyl||_{L^{\frac{n}{2}}}||h||_{L^p}^2,
		\end{align*}
		where $p=\frac{2n}{n-2}$. In the last two steps we have used
		equation \eqref{ineq weyl h} and the sharp Sobolev inequality \eqref{sharp sob}, and Kato's inequality
		\[\abs{\nabla\abs{h}}\leq\abs{\nabla h},\]
		in fact, \eqref{sharp sob} extends to compactly supported functions in $H^2_1(M)$ and since $\abs{h}$ is a Lipschitz function Kato's inequality holds almost everywhere implying the validity of \eqref{sharp sob} for $h$ (\cite[Proposition 3.49]{Aubin}, \cite[Lemma 7.7]{Lee1}). Therefore, the manifold is stable if the right-hand side is non-negative, that is equivalent to
		\begin{align*}
			||\weyl||_{L^{\frac{n}{2}}}\leq\frac{1}{2}D(n)^{-2}\sqrt{\frac{n(n+1)}{(n+2)(n-1)}}.
		\end{align*}
	\end{proof}
	Comparing the two inequalities we have found for the Weyl functional we get a small gap where the stability theorem holds. Note that, the proof of Theorem \ref{stability thm intro} holds also in dimension four; however there is no gap where we have the validity of \eqref{stability} for $n=4$,  since the constant found in Theorem \ref{isolation thm intro} is larger. For $q=\frac{1}{2}$ we have
	\begin{align*}
		D(4)^{-2}\frac{1}{2}\sqrt{\frac{10}{9}}=
		D(n)^{-2}\frac{1}{2}\sqrt{\frac{n(n+1)}{(n+2)(n-1)}}\leq D(n)^{-2}\frac{4}{3}C(n)^{-1}= D(4)^{-2}\frac{16}{3\sqrt{6}}.
	\end{align*}
	\subsection{Stability results \emph{via} Gauss-Bonnet formula for ALE six-dimensional manifolds} \label{s-alesix}
		Let $(M,g)$ be Riemannian manifold, given a local orthonormal coframe $\set{\theta^i}_{i=1}^6$, there exists a unique family of one-forms $\omega_{ij}$ (known as \emph{Levi-Civita connection forms}), determined by the \emph{first structure equations}
	\begin{align*}
		d\theta^i=-\omega_{ij}\wedge\theta^j;\quad\omega_{ij}=-\omega_{ji}.
	\end{align*}
	We now introduce a family of $2$-forms, the \emph{curvature forms} $\set{\Omega_{ij}}$, defined \emph{via} the \emph{second structure equation}
	\begin{align*}
		\Omega_{ij}=d\omega_{ij}-\omega_{ik}\wedge\omega_{ki}.
	\end{align*}
	Using the basis $\set{\theta^i\wedge\theta^j}$ of the bundle of $2$-forms, they can be written as
	\begin{align*}
		\Omega_{ij}=\frac{1}{2}R_{ijkl}\theta^k\wedge\theta^l,
	\end{align*}
	where $R_{ijkl}$ are the coefficients of the Riemann curvature tensor (for a more accurate discussion see e.g., \cite[Chapter 1]{CMbook}).
	Given a compact Riemannian manifold of dimension $n=2m$, the generalized Gauss-Bonnet formula is given by
	\begin{align*}
		\chi(M)=\int_M\Omega dV_g,
	\end{align*} 
	where
	\begin{align*}
		\Omega=\frac{(-1)^m}{2^{2m}\pi^mm!}\eps_{i_1...i_n}\Omega_{i_1i_2}\wedge...\wedge\Omega_{i_{n-1}i_n},
	\end{align*}
	and we are adopting Einstein's summation convention. It is possible to rewrite $\Omega$ in terms of the coefficients of the Riemann curvature tensor, in fact, since
	$\Omega_{ij}=\frac{1}{2}R_{ijkl}\theta^k\wedge\theta^l$, we have
	\begin{align*}
		\chi(M)=\frac{(-1)^m}{2^{3m}\pi^mm!}\int_M\eps_{i_1...i_n}\eps_{j_1...j_n}R_{i_1i_2j_1j_2}...R_{i_{n-1}i_n...j_{n-1}j_n}dV_g.
	\end{align*}
	More in general, for an $n$-dimensional manifold we denote $SM$ the $(2n-1)$-dimensional manifold of unit tangent vectors. For each point of the fiber, $(x,\xi)\in SM$, we associate to it an orthonormal coframe $\theta^i$ of $M$, such that $\theta^1$ is the dual of $\xi$. We then consider the associated Levi-Civita connection forms $\omega_{ij}$ and the curvature forms $\Omega_{ij}$. Let
	\begin{align*}
		\Pi=\frac{1}{\pi^m}\sum_{\lambda=0}^{m-1}(-1)^{\lambda}\frac{1}{1\cdot3...(2m-2\lambda-1)2^{m+\lambda}\lambda!}\Phi_{\lambda},
	\end{align*}
	where
	\begin{align*}
		\Phi_{\lambda}=\eps_{i_1...i_{n-1}}\Omega_{i_1i_2}\wedge...\wedge\Omega_{i_{2\lambda-1}i_{2\lambda}}\wedge\omega_{i_{2\lambda+1}n}\wedge...\wedge\omega_{i_{n-1}n}.
	\end{align*}
	This form was first introduced by Chern in \cite{Chern} and it is an $(n-1)$-form on $SM$. When we consider a Riemannian manifold with boundary, the generalized Gauss-Bonnet formula is given by (see e.g. \cite{Chern, DaiWei, Rosenberg, Zhu} for more details)
	\begin{align*}
		\chi(M)=\int_M\Omega+\int_{\partial M}\nu^*\Pi,
	\end{align*}
	where $\nu^*$ is the section of $SM$ over the boundary given by the outward unit normal and it can be rewritten as (see \cite{Rosenberg} for a more detailed discussion)
	\begin{align*}
		\chi(M)=\int_M\Omega+\int_{\partial M}\Pi.
	\end{align*}
	In particular, on a six dimensional Einstein manifold with boundary we have
	\begin{align*}
		384\pi^3\chi(M)=&\int_M\pa{\frac{2}{75}S^3-2\abs{\nabla\weyl}^2-\frac{7}{15}S\abs{\weyl}^2+6W_{ijkl}W_{klpq}W_{pqij}}dV_g\\
		&+\frac{1}{\pi^3}\int_{\partial M}\eps_{i_1i_2i_3i_4i_5}\Bigg(\frac{1}{15\cdot2^3}\omega_{i_16}\wedge...\wedge\omega_{i_56}-\frac{1}{3\cdot2^4}\Omega_{i_1i_2}\wedge\omega_{i_36}\wedge\omega_{i_46}\wedge\omega_{i_56}\\
		&+\frac{1}{2^6}\Omega_{i_1i_1}\wedge\Omega_{i_3i_4}\wedge\omega_{i_56}\Bigg)\\
		=:&\int_M\pa{\frac{2}{75}S^3-2\abs{\nabla\weyl}^2-\frac{7}{15}S\abs{\weyl}^2+6W_{ijkl}W_{klpq}W_{pqij}}dV_g+I(\partial M).
	\end{align*}
	\begin{rem}
		On ALE Ricci-flat manifolds 
		\begin{align*}
			I(\partial M)=I(\mathbb{S}^5/\Gamma),
		\end{align*}
		as a consequence, this quantity depends on $\frac{1}{\Gamma}$ and on the value of the integral on the sphere (analogously to dimension $4$, \cite{Nakajima}).
	\end{rem}
	\begin{proof}[Proof of Theorem \ref{t-ec}]
		When $(M,g)$ is ALE and Ricci-flat, the Gauss-Bonnet formula for manifolds with boundary in dimension $6$, reduces to
		\begin{align*}
			384\pi^3\chi(M)&=\int_M\pa{-2\abs{\nabla W}^2+6W_{ijkl}W_{klpq}W_{pqij}}dV_g+I(\mathbb{S}^5/\Gamma)\\
			&=-2||\nabla\weyl||_{L^2}^2+6\int_MW_{ijkl}W_{klpq}W_{pqij}dV_g+I(\mathbb{S}^5/\Gamma).
		\end{align*}
		Using Sobolev and Kato inequalities, we deduce
		\begin{align*}
			384\pi^3\chi(M)\leq-\frac{14}{5}D(6)^{-2}||\weyl||_{L^3}^2+6\int_MW_{ijkl}W_{klpq}W_{pqij}dV_g+I(\mathbb{S}^5/\Gamma),
		\end{align*}
		that is
		\begin{align}\label{ineq with euler char 1}
			||\weyl||_{L^3}^2&\leq\frac{5}{14}D(6)^{2}\pa{6\int_MW_{ijkl}W_{klpq}W_{pqij}dV_g+I(\mathbb{S}^5/\Gamma)-384\pi^3\chi(M)}\\
			&=\frac{5}{14}D(6)^{2}\pa{6\int_M\mathrm{tr}(\weyl^3)dV_g+I(\mathbb{S}^5/\Gamma)-384\pi^3\chi(M)},\notag
		\end{align}
		then, by \eqref{t-eulc} we deduce that $(M,g)$ satisfies the stability criterion of Theorem \ref{stability thm intro} and hence it is stable.
	\end{proof}
	\section{Stability results for the Weyl functional on PE manifolds}\label{Sec PE}
	In this section, we are interested in stability results for PE manifolds with positive Yamabe invariant. In particular we are going to provide a stability criterion involving $Y(M,[g])$ and $||\weyl||_{L^p}$. The main tool we use here is the inequality 
	\begin{align}\label{nabla h AH}
		||\nabla h||_{L^2}^2\geq\lambda_0||h||_{L^2}^2,
	\end{align}
	where $\lambda_0$ is the bottom of the $L^2$-spectrum for the Laplace operator; note that \eqref{nabla h AH} holds for every compactly supported tensor $T$ by Lemma 7.7 of \cite{Lee1}. It was shown by Mazzeo (see \cite{Mazzeo}) that for AH metrics $\lambda_0\leq\frac{(n-1)^2}{4}$. In \cite{LeeSpectrum}, Lee proved that on PE manifolds whose conformal infinity has non-negative Yamabe invariant $\lambda_0=\frac{(n-1)^2}{4}$. 
	Moreover, from the definition of $Y(M,[g])$, for every $u\in C^{\infty}_0$, $u\neq 0$, we have the validity of
	\begin{align*}
		||u||_{L^{\frac{2n}{n-2}}}^2Y(M,[g])\leq\frac{4(n-1)}{n-2}||\nabla u||_{L^2}^2+\int_MSu^2dV_g,
	\end{align*}
	which, on PE manifolds, rewrites as
	\begin{align}\label{sobolev yamabe}
		||u||_{L^{\frac{2n}{n-2}}}^2Y(M,[g])\leq\frac{4(n-1)}{n-2}||\nabla u||_{L^2}^2-n(n-1)||u||_{L^2}^2.
	\end{align}
	\begin{rem}
		Inequality \eqref{sobolev yamabe} extends continuously to compactly supported functions $u\in H^2_1(M)$, moreover it holds also if we replace $u$ by any compactly supported tensor $T$ because $\abs{T}$ is a Lipschitz function and Kato's inequality 
		\[\abs{\nabla \abs{T}}\leq\abs{\nabla T}\]
		holds almost everywhere.
	\end{rem}
	\begin{proof}[Proof of Theorem \ref{t-stability PE}] 
		Note that, by definition of AH, $(M,g)$ is conformal to the interior of a compact manifold with boundary, therefore $||\weyl||_{L^\frac{n}{2}}$ is finite since it is conformally invariant. Hence, $\abs{\abs{\weyl}}_{L^p}$ is finite for every $p\in\sq{\frac{n}{2},\infty}$.\\
		Let $h$ be a $TT$ tensor. The Einstein operator, $\Delta_E$, satisfies the Bochner formula \eqref{bochner e oper D2}, which we recall here:
		\begin{equation}\label{DED2}
			(\Delta_Eh,h)_{L^2}=||D_2h||^2_{L^2}+(n-2)||h||^2_{L^2}-(\mathring{W}h,h)_{L^2}
		\end{equation}
		and, by its definition, we have
		\begin{equation}\label{DE}
			(\Delta_Eh,h)_{L^2}=||\nabla h||_{L^2}^2-2||h||_{L^2}^2-2(\mathring{\weyl}h,h).
		\end{equation}
		Let $\alpha\in[0,1]$, combining \eqref{DED2} and \eqref{DE} we have
		\begin{equation}\label{DEalpha}
			(\Delta_Eh,h)_{L^2}\geq\alpha||\nabla h||_{L^2}^2+\pa{(n-2)(1-\alpha)-2\alpha}||h||_{L^2}^2-(2\alpha+(1-\alpha))(\mathring{\weyl}h,h)
		\end{equation}
		and using \eqref{ineq weyl h point} and H\"{o}lder inequality, we obtain
		\begin{align*}
			(\Delta_E h,h)_{L^2}
			\geq&\alpha||\nabla h||_{L^2}^2+\pa{n(1-\alpha)-2}||h||_{L^2}^2-(\alpha+1)s(n)||\weyl||_{L^p}||h||_{L^q}^2,
		\end{align*}
		where $p\in\sq{\frac{n}{2},+\infty},\,q=\frac{2p}{p-1}\in\sq{2,\frac{2n}{n-2}}$. 
		By $L^q$-interpolation, we deduce
		\begin{align*}
			(\Delta_E h,h)_{L^2}
			\geq&\alpha||\nabla h||_{L^2}^2+\pa{n(1-\alpha)-2}||h||_{L^2}^2-(\alpha+1)s(n)||\weyl||_{L^p}||h||_{L^2}^{2t}||h||_{L^{\frac{2n}{n-2}}}^{2(1-t)}\\
			=&\alpha||\nabla h||_{L^2}^2+\pa{n(1-\alpha)-2}||h||_{L^2}^2-(\alpha+1)s(n)||\weyl||_{L^p}\lambda||h||_{L^2}^{2t}\lambda^{-1}||h||_{L^{\frac{2n}{n-2}}}^{2(1-t)},
		\end{align*}
		for a positive constant $\lambda$ and $t$ satisfying
		\[\frac{1}{q}=\frac{t}{2}+\frac{n-2}{2n}(1-t),\]
		that is
		\[t=1-\frac{n}{2p}.\]
		With an abuse of notation, we set
		\begin{align*}
				t\lambda^{\frac{1}{t}}=0,\quad \text{ for } t=0,\quad\text{ and }\quad
				(1-t)\lambda^{\frac{1}{t-1}}=0,\quad \text{ for } t=1,
		\end{align*}
		in order to include the cases $t=0$, $t=1$ in the following inequalities and computations, recovering the cases $p=\frac{n}{2}$ and $p=+\infty$.
		Exploiting Young inequality we deduce
		\begin{align*}
			(\Delta_E h,h)_{L^2}
			\geq&\alpha||\nabla h||_{L^2}^2+\pa{n(1-\alpha)-2}||h||_{L^2}^2\\
			&-(\alpha+1)s(n)||\weyl||_{L^p}\pa{t\lambda^{\frac{1}{t}}||h||_{L^2}^{2}+(1-t)\lambda^{\frac{1}{t-1}}||h||_{L^{\frac{2n}{n-2}}}^{2}}.
		\end{align*}
		Moreover, using Sobolev inequality (see formula \eqref{sobolev yamabe}) and \eqref{nabla h AH} we get
		\begin{align*}
			(\Delta_E h,h)_{L^2}
			\geq&\pa{\alpha-\pyamw(\alpha+1)(1-t)\lambda^{\frac{1}{t-1}}\frac{4(n-1)}{n-2}s(n)}||\nabla h||_{L^2}^2\\
			+&\sq{n(1-\alpha)-2+||\weyl||_{L^p}\pa{(1-t)\lambda^{\frac{1}{t-1}}\frac{n(n-1)}{Y(M,[g])}-t\lambda^{\frac{1}{t}}}(\alpha+1)s(n)}||h||_{L^2}^2\\
			\geq&\pa{\alpha(1-\tau)-\pyamw(\alpha+1)(1-t)\lambda^{\frac{1}{t-1}}\frac{4(n-1)}{n-2}s(n)}||\nabla h||_{L^2}^2\\
			+&\sq{n(1-\alpha)-2+\alpha\tau\lambda_0+||\weyl||_{L^p}\pa{(1-t)\lambda^{\frac{1}{t-1}}\frac{n(n-1)}{Y(M,[g])}-t\lambda^{\frac{1}{t}}}(\alpha+1)s(n)}||h||_{L^2}^2,
		\end{align*}
		where $\tau\in[0,1]$. For the sake of simplicity, we will write
		\[
		(\Delta_E h,h)_{L^2}
		\geq A||\nabla h||_{L^2}^2+B||h||_{L^2}^2,
		\]
		where 
		\begin{align*}
			A:=&\alpha(1-\tau)-\pyamw(\alpha+1)(1-t)\lambda^{\frac{1}{t-1}}\frac{4(n-1)}{n-2}s(n),\\
			B:=&n(1-\alpha)-2+\alpha\tau\lambda_0+||\weyl||_{L^p}\pa{(1-t)\lambda^{\frac{1}{t-1}}\frac{n(n-1)}{Y(M,[g])}-t\lambda^{\frac{1}{t}}}(\alpha+1)s(n).
		\end{align*}
		Let us take $\alpha$ and $\lambda$ such that
		\[\frac{\alpha}{\alpha+1}\lambda^{\frac{1}{1-t}}=\frac{(1-t)}{2}Y(M,[g])^{-t}\sq{8(n-1)}^t.\]
		In particular, set
		\begin{align}\label{e-alphalambda}
			\begin{cases}
				\lambda^{\frac{1}{1-t}}&= Y(M,[g])^{-t}\sq{8(n-1)}^t,\\
				\frac{\alpha}{\alpha+1}&=\frac{1-t}{2}=\frac{n}{4p}, \quad \text{i.e.,}\quad\alpha=\frac{n}{4p-n}\in\sq{0,1}
			\end{cases}
		\end{align}
		and
		\begin{align}\label{e-tau}
			\tau=1-\pyamw\frac{(\alpha+1)}{\alpha}\lambda^{\frac{1}{t-1}}\frac{4(n-1)}{n-2}s(n)(1-t).
		\end{align}
		Observe that, by \eqref{e-alphalambda} and \eqref{stab Lp}, $\tau\in[0,1]$, in fact
		\begin{align*}
			\tau\geq0\quad\Leftrightarrow\quad ||\weyl||_{L^p}\leq Y(M,[g])^{\frac{n}{2p}} \frac{(n-2)}{\sq{8(n-1)}^{\frac{n}{2p}}}s(n)^{-1}.
		\end{align*}
		Furthermore, $A=0$ by construction. On the other hand, Sobolev inequality \eqref{sobolev yamabe} yields
		\begin{equation}\label{simil poincare}
			||\nabla h||_{L^2}^2\geq\frac{n(n-2)}{4}||h||_{L^2}^2.
		\end{equation}
		Therefore $\lambda_0\geq\frac{n(n-2)}{4}$ and $B$ satisfies
		\begin{align*}
			B=&n(1-\alpha)-2+\alpha\lambda_0\pa{1-\pyamw\frac{(\alpha+1)}{\alpha}\lambda^{\frac{1}{t-1}}\frac{4(n-1)}{n-2}s(n)(1-t)}\\
			&+||\weyl||_{L^p}\pa{(1-t)\lambda^{\frac{1}{t-1}}\frac{n(n-1)}{Y(M,[g])}-t\lambda^{\frac{1}{t}}}(\alpha+1)s(n)\\
			\geq&n(1-\alpha)-2+\alpha\frac{n(n-2)}{4}\pa{1-\pyamw\frac{(\alpha+1)}{\alpha}\lambda^{\frac{1}{t-1}}\frac{4(n-1)}{n-2}s(n)(1-t)}\\
			&+||\weyl||_{L^p}\pa{(1-t)\lambda^{\frac{1}{t-1}}\frac{n(n-1)}{Y(M,[g])}-t\lambda^{\frac{1}{t}}}(\alpha+1)s(n)\\
			=& n(1-\alpha)-2+\alpha\frac{n(n-2)}{4}-||\weyl||_{L^p}\pa{1-\frac{n}{2p}}\lambda^{\frac{1}{t}}(\alpha+1)s(n)
		\end{align*}
		By \eqref{e-alphalambda},
		\[\lambda=Y(M,[g])^{-t(1-t)}\sq{8(n-1)}^{t(1-t)}\]
		and
		\[\lambda^{\frac{1}{t}}=Y(M,[g])^{t-1}\sq{8(n-1)}^{(1-t)}.\]
		Thus,
		\begin{align*}
			B&\geq n(1-\alpha)-2+\alpha\frac{n(n-2)}{4}-(\alpha+1)t||\weyl||_{L^p}Y(M,[g])^{t-1}\sq{8(n-1)}^{(1-t)}s(n).
		\end{align*}
		Therefore, it follows by \eqref{stab Lp}
		\begin{align*}
			B\geq& n(1-\alpha)-2+\alpha\frac{n(n-2)}{4}-(\alpha+1)t(n-2),
		\end{align*}
		which, by definition of $t$, is
		\begin{align*}
			B\geq& n(1-\alpha)-2+\alpha\frac{n(n-2)}{4}-(\alpha+1)\pa{1-\frac{n}{2p}}(n-2).
		\end{align*}
		Substituting $\alpha=\frac{n}{4p-n}$, we get
		\[B\geq\frac{n^3-2n^2-8n}{4(4p-n)},\]
		which is non-negative for $n\geq4$.
	\end{proof}
	\section{Stability results for compact Einstein and K\"{a}hler-Einstein manifolds}\label{Sec stability cpt}
	This section is devoted to stability results for compact Einstein and K\"{a}hler-Einstein manifolds with positive scalar curvature. The methods used in the proof of Theorem \ref{lpWcpt} and Theorem \ref{lpBcpt} are based on the that of Theorem 1.4 and 1.5 and Theorem 5.6 and 5.7 of \cite{Kroncke15}, respectively. In particular, for $p=\frac{n}{2}$ and $p=\infty$, we recover the already known results of \cite{Kroncke15}, with a slight improvement, due to the constant $s(n)$.\\
	In particular, we want to combine \eqref{bochner e oper D1} and the definition of $\Delta_E$ to obtain a result involving the $L^p$ norm of the Weyl (Bochner, respectively) tensor. One of the main tool we are going to use is the validity of Sobolev inequality, which we recall here 
	\begin{align*}
		\abs{\abs{h}}_{L^{\frac{2n}{n-2}}}^2&\leq\frac{1}{Y(M,[g])}\pa{\frac{4(n-1)}{n-2}\abs{\abs{\nabla h}}_{L^2}^2+S\abs{\abs{h}}_{L^2}^2},
	\end{align*}
	where $Y(M,[g])$ is the Yamabe invariant of $[g]$ on $M$. Since $(M,g)$ is  a Yamabe metric, we have
	\[Y(M,[g])=\mathrm{Vol}(M)^{\frac{2}{n}}S,\]
	thus,
	\begin{align}\label{sob cpt}
		\abs{\abs{h}}_{L^{\frac{2n}{n-2}}}^2\leq\mathrm{Vol}(M)^{-\frac{2}{n}}\frac{4(n-1)}{n-2}\frac{\abs{\abs{\nabla h}}_{L^2}^2}{S}+\mathrm{Vol}(M)^{-\frac{2}{n}}\abs{\abs{h}}_{L^2}^2,
	\end{align}
	\begin{proof}[Proof of Theorem \ref{lpWcpt}]
		Let $\alpha\in[0,1]$, we combine \eqref{bochner e oper D1}, which we recall here, that for every $TT$ compactly supported tensor, we have
		\begin{align*}
			(\Delta_E h,h)_{L^2}&=||D_1h||^2_{L^2}+2\frac{S}{n}\abs{\abs{h}}_{L^2}^2-4(\mathring{R}h,h)_{L^2}\\
			&=||D_1h||_{L^2}+\frac{2(n+1)}{n(n-1)}S\abs{\abs{h}}_{L^2}^2-4(\mathring{W}h,h)_{L^2}
		\end{align*}
		with the definition of $\Delta_E$ which, in terms of $\weyl$, is
		\begin{align*}
			(\Delta_E h,h)_{L^2}&=\abs{\abs{\nabla h}}_{L^2}^2+\frac{2S}{n(n-1)}\abs{\abs{h}}_{L^2}^2-2(\mathring{W}h,h)_{L^2},
		\end{align*}
		to obtain
		\begin{align}\label{stab1}
			(\Delta_E h,h)_{L^2}=&\alpha\abs{\abs{\nabla h}}_{L^2}^2+\sq{\alpha+(1-\alpha)(n+1)}\frac{2S}{n(n-1)}\abs{\abs{h}}_{L^2}^2\\
			&-2\sq{\alpha+2\pa{1-\alpha}}(\mathring{W}h,h)+(1-\alpha)\abs{\abs{D_1 h}}_{L^2}^2\notag\\
			\geq&\alpha\abs{\abs{\nabla h}}_{L^2}^2+\pa{n+1-n\alpha}\frac{2S}{n(n-1)}\abs{\abs{h}}_{L^2}^2-2\pa{2-\alpha}(\mathring{W}h,h).\notag
		\end{align}
		By \eqref{ineq weyl h point} and H\"{o}lder inequality 
		\[(\mathring{W}h,h)\leq s(n)\abs{\abs{\weyl}}_{L^p}\abs{\abs{h}}_{L^q}^2,\]
		with $q=\frac{2p}{p-1}\in\sq{2,\frac{2n}{n-2}}$.
		As for the proof of Theorem \ref{stab Lp}, with an abuse of notation, we set
		\begin{align*}
			t\lambda^{\frac{1}{t}}=0,\quad \text{ for } t=0,\quad\text{ and }\quad
			(1-t)\lambda^{\frac{1}{t-1}}=0,\quad \text{ for } t=1,
		\end{align*}
		and exploiting $L^q$ interpolation and Young inequality, we deduce
		\begin{align}\label{interpolation Y}
			(\mathring{W}h,h)\leq s(n)\abs{\abs{\weyl}}_{L^p}\abs{\abs{h}}_{L^q}^2\leq s(n)\abs{\abs{\weyl}}_{L^p}\pa{t\lambda^{\frac{1}{t}}\abs{\abs{h}}_{L^2}^2+(1-t)\lambda^{\frac{1}{t-1}}\abs{\abs{h}}^2_{L^{\frac{2n}{n-2}}}},
		\end{align}
		for $t=1-\frac{n}{2p}$ and $\lambda\in \erre\setminus\set{0}$. Take
		\[\lambda=\sq{\frac{4(n-1)}{n(n-2)}+1}^{-t(t-1)}\mathrm{Vol}(M)^{-\frac{t}{p}}.\]
		Plugging \eqref{interpolation Y} into \eqref{stab1} and using
		Sobolev inequality, \eqref{sob cpt}, we have
		\begin{align}\label{stab 2}
			(\Delta_E h,h)\geq&\alpha\abs{\abs{\nabla h}}_{L^2}^2+\pa{n+1-\alpha n}\frac{2S}{n(n-1)}\abs{\abs{h}}_{L^2}^2\\
			&-2(2-\alpha)s(n)\abs{\abs{\weyl}}_{L^p}\pa{t\lambda^{\frac{1}{t}}\abs{\abs{h}}_{L^2}^2+(1-t)\lambda^{\frac{1}{t-1}}\abs{\abs{h}}^2_{L^{\frac{2n}{n-2}}}}\notag\\
			\geq&\pa{\alpha-2(2-\alpha)s(n)(1-t)\lambda^{\frac{1}{t-1}}\frac{4(n-1)}{n-2}\frac{\abs{\abs{\weyl}}_{L^{p}}}{S}\mathrm{Vol}(M)^{-\frac{2}{n}}}\abs{\abs{\nabla h}}_{L^2}^2\notag\\
			&+\pa{(n+1-n\alpha)\frac{2S}{n(n-1)}-2(2-\alpha)s(n)\sigma(t)\abs{\abs{\weyl}}_{L^p}}\abs{\abs{h}}_{L^2}^2.\notag
		\end{align}
		where, for the sake of simplicity we denote
		\begin{equation*}
			\sigma(t)=t\lambda^{\frac{1}{t}}+(1-t)\lambda^{\frac{1}{t-1}}\mathrm{Vol}(M)^{-\frac{2}{n}}.
		\end{equation*}
		Arguing as in \cite{Kroncke15}, by \eqref{bochner e oper D1} we deduce
		\begin{align*}
			\abs{\abs{\nabla h}}_{L^2}^2&=\abs{\abs{D_1h}}_{^2}^2+2\frac{S}{n}\abs{\abs{h}}_{L^2}^2-2(\mathring{R}h,h)_{L^2}\\
			&\geq\frac{2}{n-1}S\abs{\abs{h}}_{L^2}^2-2(\mathring{W}h,h)_{L^2}\\
			&\geq\frac{2}{n-1}S\abs{\abs{h}}_{L^2}^2-2\abs{\abs{\weyl}}_{L^p}\abs{\abs{h}}_{L^q}^2,
		\end{align*}
		and using interpolation and Young inequality, we deduce
		\begin{align*}
			\abs{\abs{\nabla h}}_{L^2}^2&\geq\frac{2}{n-1}S\abs{\abs{h}}_{L^2}^2-2s(n)\abs{\abs{\weyl}}_{L^p}\pa{t\lambda^{\frac{1}{t}}\abs{\abs{h}}_{L^2}^2+(1-t)\lambda^{\frac{1}{t-1}}\abs{\abs{h}}^2_{L^{\frac{2n}{n-2}}}}.		
		\end{align*}
		By Sobolev inequality it follows
		\begin{align*}
			\abs{\abs{\nabla h}}_{L^2}^2
			\geq&\sq{\frac{2}{n-1}S-2s(n)\sigma(t)\abs{\abs{\weyl}}_{L^p}}\abs{\abs{h}}_{L^2}^2\\
			&-s(n)(1-t)\lambda^{\frac{1}{t-1}}\frac{8(n-1)}{n-2}\frac{\abs{\abs{\weyl}}_{L^p}}{S}\mathrm{Vol}(M)^{-\frac{2}{n}}\abs{\abs{\nabla h}}_{L^2}^2.
		\end{align*}
		Therefore, $\abs{\abs{\nabla h}}_{L^2}^2$ can be estimated by
		\begin{align}\label{nabla h}
			\abs{\abs{\nabla h}}_{L^2}^2\geq&\pa{1+s(n)(1-t)\lambda^{\frac{1}{t-1}}\frac{8(n-1)}{n-2}\frac{\abs{\abs{\weyl}}_{L^p}}{S}\mathrm{Vol}(M)^{-\frac{2}{n}}}^{-1}\cdot\\
			&\cdot\sq{\frac{2}{n-1}S-2s(n)\sigma(t)\abs{\abs{\weyl}}_{L^p}}\abs{\abs{h}}_{L^2}^2.\notag
		\end{align}
		Combining \eqref{stab 2} and \eqref{nabla h}, we deduce
		\begin{align*}
			(\Delta_E h,h)\geq&\Bigg\{\pa{\alpha-(2-\alpha)s(n)(1-t)\lambda^{\frac{1}{t-1}}\frac{8(n-1)}{n-2}\frac{\abs{\abs{\weyl}}_{L^{p}}}{S}\mathrm{Vol}(M)^{-\frac{2}{n}}}\cdot\\
			&\cdot\pa{1+s(n)(1-t)\lambda^{\frac{1}{t-1}}\frac{8(n-1)}{n-2}\frac{\abs{\abs{\weyl}}_{L^p}}{S}\mathrm{Vol}(M)^{-\frac{2}{n}}}^{-1}\cdot\\
			&\cdot\sq{\frac{2}{n-1}S-2s(n)\sigma(t)\abs{\abs{\weyl}}_{L^p}}\\
			&+\pa{(n+1-n\alpha)\frac{2S}{n(n-1)}-2(2-\alpha)s(n)\sigma(t)\abs{\abs{\weyl}}_{L^p}}\Bigg\}\abs{\abs{h}}_{L^2}^2.
		\end{align*}
		Therefore, $(M,g)$ is stable (strictly stable) if the left-hand side of this inequality is non-negative (positive). 
		It is easy to check that this holds if 
		\begin{align*}
			\set{\frac{4(n-1)}{n(n-2)}(1-t)\lambda^{\frac{1}{t-1}}\mathrm{Vol}(M)^{-\frac{2}{n}}+\sigma(t)}s(n)\abs{\abs{\weyl}}_{L^p}\leq\frac{S}{2}\frac{n+1}{n(n-1)},
		\end{align*}
		which is equivalent to  
		\begin{align*}
			\set{\pa{\frac{4(n-1)}{n(n-2)}+1}(1-t)\lambda^{\frac{1}{t-1}}\mathrm{Vol}(M)^{-\frac{2}{n}}+t\lambda^{\frac{1}{t}}}s(n)\abs{\abs{\weyl}}_{L^p}\leq\frac{S}{2}\frac{n+1}{n(n-1)}.
		\end{align*}
		due to the definition of $\sigma(t)$. Moreover, substituting the expression of $t$ and $\lambda$ we obtain
		\[	\abs{\abs{\weyl}}_{L^p}\leq\frac{S}{2}\frac{n+1}{n(n-1)}\pa{\frac{4(n-1)}{n(n-2)}+1}^{-\frac{n}{2p}}s(n)^{-1}\mathrm{Vol}(M)^{\frac{1}{p}}.\]
	\end{proof}
	Adapting the argument of Theorem \ref{lpWcpt} to K\"{a}hler-Einstein manifold we obtain:
	\begin{proof}[Proof of Theorem \ref{lpBcpt}]
		On a K\"{a}hler-Einstein manifold, the Riemann curvature tensor can be written as 
		\[R_{ijkl}=B_{ijkl}+\frac{S}{n(n+2)}\pa{\delta_{ik}\delta_{jl}-\delta_{il}\delta_{jk}+J^k_iJ^l_j-J^l_iJ^k_j+2J^j_iJ^l_k},\]
		with respect to a local orthonormal frame.
		As a consequence, a straightforward computation combined with the Cauchy-Schwarz inequality, implies
		\begin{align}\label{riemboc}
			R_{ijkl}h_{ik}h_{jl}\leq B_{ijkl}h_{ik}h_{jl}+2\frac{S}{n(n+2)}\abs{h}^2.
		\end{align}
		Consider $\alpha\in[0,1]$, for every $TT$ tensor $h$ with compact support, we combine
		\begin{align*}
			(\Delta_E h,h)_{L^2}&=||D_1h||^2_{L^2}+2\frac{S}{n}\abs{\abs{h}}_{L^2}^2-4(\mathring{R}h,h)_{L^2}\\
			&\geq\frac{2(n-2)}{n(n+2)}S\abs{\abs{h}}_{L^2}^2-4(\mathring{B}h,h)_{L^2}
		\end{align*}
		with
		\begin{align*}
			(\Delta_E h,h)_{L^2}&=\abs{\abs{\nabla h}}_{L^2}^2-2(\mathring{R}h,h)_{L^2}\\
			&\geq\abs{\abs{\nabla h}}_{L^2}^2-\frac{4S}{n(n+2)}\abs{\abs{h}}_{L^2}^2-2(\mathring{B}h,h)_{L^2}
		\end{align*}
		to obtain
		\begin{align}\label{stabB1}
			(\Delta_E h,h)_{L^2}&\geq\alpha\abs{\abs{\nabla h}}_{L^2}^2+\frac{2S}{n(n+2)}(n-2-n\alpha)\abs{\abs{h}}_{L^2}^2-2(2-\alpha)(\mathring{B}h,h)_{L^2}.
		\end{align}
		Note that, since $\boc$ can be regarded as a symmetric linear operator (due to its symmetries) on symmetric $(0,2)$-tensors, we have the validity of the following pointwise inequality: let $\eta$ be the largest eigenvalue of $\boc$, then
		\[B_{ijkl}h_{ik}h_{jl}\leq \eta\abs{h}^2\leq s(n)\abs{\boc}\abs{h}^2.\]
		Arguing as in Theorem \ref{lpWcpt} and using the same notation, we deduce
		\begin{align}\label{estimB}
			(\mathring{B}h,h)_{L^2}\leq s(n)\abs{\abs{\boc}}_{L^p}\Bigg[\sigma(t)\abs{\abs{h}}_{L^2}^2+(1-t)\lambda^{\frac{1}{t-1}}\frac{4(n-1)}{n-2}\frac{\mathrm{Vol}(M)^{-\frac{2}{n}}}{S}\abs{\abs{\nabla h}}_{L^2}^2\Bigg],
		\end{align}
		where $t=1-\frac{n}{2p}$,
		\[\sigma(t)=t\lambda^{\frac{1}{t}}+(1-t)\lambda^{\frac{1}{t-1}}\mathrm{Vol}(M)^{-\frac{2}{n}}\]
		and we take
		\[\lambda=\sq{\frac{4(n-1)}{n(n-2)}+1}^{-t(t-1)}\mathrm{Vol}(M)^{-\frac{t}{p}}.\]
		Adapting the same argument of Theorem \ref{lpWcpt}, we deduce that $\abs{\abs{\nabla h}}_{L^2}^2$ can be estimated by
	\begin{align}\label{estim nab h b}
		\abs{\abs{\nabla 	h}}^2\geq&\pa{1+(1-t)\lambda^{\frac{1}{t-1}}\frac{8(n-1)}{n-2}s(n)\frac{\mathrm{Vol(M)^{-\frac{2}{n}}}}{S}\abs{\abs{\boc}}_{L^p}}^{-1}\sq{2\frac{S}{n+2}-2s(n)\sigma(t)\abs{\abs{\boc}}_{L^p}}\abs{\abs{h}}_{L^2}^2.
	\end{align}
	Thus, substituting \eqref{estimB} and \eqref{estim nab h b} into \eqref{stabB1} we get
	\begin{align*}
		(\Delta_E h,h)_{L^2}\geq&\Bigg\{\pa{1+(1-t)\lambda^{\frac{1}{t-1}}\frac{8(n-1)}{n-2}s(n)\frac{\mathrm{Vol(M)^{-\frac{2}{n}}}}{S}\abs{\abs{\boc}}_{L^p}}^{-1}\sq{2\frac{S}{n+2}-2s(n)\sigma(t)\abs{\abs{\boc}}_{L^p}}\cdot\\
		&\cdot\sq{\alpha-(2-\alpha)(1-t)\lambda^{\frac{1}{t-1}}\frac{8(n-1)}{n-2}s(n)\frac{\mathrm{Vol}(M)^{-\frac{2}{n}}}{S}\abs{\abs{\boc}}_{L^p}}\\
		&+\sq{\frac{2S}{n(n+2)}(n-2-n\alpha)-2(2-\alpha)s(n)\sigma(t)\abs{\abs{\boc}}_{L^p}}\Bigg\}\abs{\abs{h}}_{L^2}^2.
	\end{align*}
	Therefore, $(M,g)$ is stable (strictly stable) if the left-hand side of this inequality is non-negative (positive).
	It is easy to check that this holds if 
	\begin{align*}
		\set{\frac{4(n-1)}{n(n-2)}(1-t)\lambda^{\frac{1}{t-1}}\mathrm{Vol}(M)^{-\frac{2}{n}}+\sigma(t)}s(n)\abs{\abs{\boc}}_{L^p}\leq\frac{S}{2}\frac{n-2}{n(n+2)}
	\end{align*}
	which is equivalent to 
	\begin{align*}
		\set{\pa{\frac{4(n-1)}{n(n-2)}+1}(1-t)\lambda^{\frac{1}{t-1}}\mathrm{Vol}(M)^{-\frac{2}{n}}+t\lambda^{\frac{1}{t}}}s(n)\abs{\abs{\boc}}_{L^p}\leq\frac{S}{2}\frac{n-2}{n(n+2)}
	\end{align*}
	due to the definition of $\sigma(t)$. Moreover, substituting the expression of $t$ and $\lambda$ we obtain
	\[\abs{\abs{\boc}}_{L^p}\leq\frac{S}{2}\frac{n-2}{n(n+2)}\pa{\frac{4(n-1)}{n(n-2)}+1}^{-\frac{n}{2p}}s(n)^{-1}\mathrm{Vol}(M)^{\frac{1}{p}}.\]		
	\end{proof}
	On the other hand, there are isolation theorems for the for the Weyl and Bochner tensor. In particular, when $n=4$, Gursky and LeBrun proved
	\begin{theorem}[Theorem 1 of \cite{GurskyLeBrun}]
		Let $(M,g)$ be a compact oriented Einstein manifold of dimension $4$ with $S>0$ and $\weyl^+\not\equiv 0$. Then
		\[\int_M\abs{\weyl^+}^2dV_g\geq\frac{1}{6}\int_MS^2dV_g\]
		with equality if and only if $\nabla \weyl^+\equiv 0$.
	\end{theorem}
	More in general, for $n\geq4$ we have
	\begin{cor}[Corollary 1.2 of \cite{BCDM}]
		Let $(M,g)$ be a compact (conformally) Einstein manifold of dimension
		$n\geq 4$ with positive Yamabe invariant. Then, either
		$(M,g)$ is locally conformally flat or
		\begin{equation} \label{hebeyimprove}
			Y(M,[g])\leq A(n)\pa{\int_M\abs{\weyl}^{\frac{n}{2}}
				d\mu_g}^{\frac{2}{n}},
		\end{equation}
		where $A(4)=\sqrt{6}$, $A(5)=\frac{64}{3\sqrt{10}}$,
		$A(6)=\sqrt{210}$ and
		$A(n)=\frac{5}{2}n$ for
		$n\geq 7$;
		if $n=5$, \eqref{hebeyimprove} is a strict inequality.
	\end{cor}
	Note that when $n\geq7$, $A(n)$ can be slightly improved (using \cite[Corollary 2.5]{Huisken}), obtaining
	\[A(n)=n\pa{2+\frac{1}{2}\frac{n(n-1)-4}{\sqrt{n(n-1)(n+1)(n-2)}}}.\]
	This results, combined with Corollary \ref{c-boc} (see Section \ref{sec Boc} for more details), rule out the stability criterion for the $L^{\frac{n}{2}}$-norm of Theorem \ref{lpWcpt} when $n=4,5,6,7$ and $n=4,6$, respectively.
	\section{Isolation results for the Bochner and contact Bochner tensors}\label{sec Boc}
	In this section we provide pinching results for the Bochner tensor of a K\"{a}hler-Einstein manifold in general dimensions $n=2m$. The key tool we are going to use is a Bockner-Weitzenb\"{o}ck formula for the Bochner tensor of a K\"{a}hler-Einstein metric together with a suitable modification of the Yamabe invariant. The technique used is the same one of \cite{BCDM}, coupled with the definition of the Bochner tensor. Finally, we will exploit an isolation result for the contact Bochner tensor of a Sasaki-Einstein manifold. Our result provides an improvement of the constant which was found by Itoh and Kobayashi in \cite{ItohKobayashi} and it is direct consequence of \cite[Theorem B]{ItohKobayashi} combined with the constants obtained in \cite{BCDM} for the Weyl tensor, $\weyl$.
	\subsection{K\"{a}hler-Einstein manifolds}
	We recall that the Bochner curvature tensor of a K\"{a}hler manifold is the $(0,4)$-tensor		
	whose components with respect to complex local coordinates are
	\begin{align*}
		B_{i\ol{j}k\ol{l}}=&R_{i\ol{j}k\ol{l}}-\frac{1}{m+2}\pa{R_{i\ol{j}}g_{k\ol{l}}+R_{k\ol{j}}g_{i\ol{l}}+R_{k\ol{l}}g_{i\ol{j}}+R_{i\ol{l}}g_{k\ol{j}}}+\frac{R}{(m+1)(m+2)}\pa{g_{i\ol{j}}g_{k\ol{l}}+g_{k\ol{j}}g_{i\ol{l}}}.
	\end{align*}
	In particular, when $(M,g)$ is K\"{a}hler-Einstein, it encodes all the curvature information, indeed
	\[
	R_{i\ol{j}k\ol{l}}=B_{i\ol{j}k\ol{l}}+\frac{R}{m(m+1)}\pa{g_{i\ol{j}}g_{k\ol{l}}+g_{k\ol{j}}g_{i\ol{l}}}
	\]
	and it satisfies the Bochner-Weitzenb\"{o}ck formula \eqref{bw for b}.
		\begin{rem}
		Although the Bochner tensor was introduced as an Hermitian analogue of the Weyl tensor, in general (when it is considered on an Hermitian manifold), it is not invariant under conformal transformations of the metric. However, in \cite{TricerriVanhecke}, Tricerri and Vanhecke studied the decomposition of the space of curvature tensors over an
		Hermitian vector space under the unitary group's action and introduced a generalization of the Bochner tensor, which they proved to be conformal invariant. This generalized tensor coincides with the standard Bochner tensor on K\"{a}hler manifolds. Moreover, in the four-dimensional case, the Bochner tensor of a K\"{a}hler metric coincides with the self-dual part of the Weyl tensor (\cite{Itoh, TricerriVanhecke}). 
	\end{rem}
	As a consequence, we have that on a K\"{a}hler manifold, the $L^{\frac{n}{2}}$-norm of $\operatorname{B}$ is invariant under conformal change of the metric, in fact let $\tilde{g}=e^{2f}g$, $f\in C^{\infty}(M)$, be a conformal change of the metric $g$;
	where $\tilde{\boc}$ denotes the (Tricerri-Vanhecke conformally invariant) Bochner relative to $\tilde{g}$, and since
	\[dV_{\tilde{g}}=e^{nf}dV_g,\]
	we deduce that
	\[||\tilde{\operatorname{B}}||_{L^{\frac{n}{2}}(\tilde{g})}=||\operatorname{B}||_{L^{\frac{n}{2}}(g)}.\]
	We point out that being K\"{a}hler is not a conformal invariant property, therefore, the new tensor $\tilde{\boc}$ is the modification of the Bochner tensor of $\tilde{g}$ introduced in \cite{TricerriVanhecke}, which is conformally invariant and it coincides with the Bochner tensor of the K\"{a}hler metric. As a consequence it allows us adapt the technique of Theorem 1.1 of \cite{BCDM}, obtaining a sharp pinching result on a K\"{a}hler-Einstein metric. 
%
%
%
	\begin{proof}[Proof of Theorem \ref{thm bochner intro}]
		Assume that $\abs{\boc}\neq 0$; to find integral estimates for $\Delta \abs{\operatorname{B}}^q$, we exploit the same technique used in \cite{BCDM}, which is similar to those used in \cite{BourCarron} and \cite{Gursky200}. Let us define the following modification of the Yamabe invariant
		\begin{align*}
			\ol{Y}^t(M,[g]):&=\inf_{\tilde{g}\in[g]}\mathrm{Vol}_{\tilde{g}}(M)^{-\frac{n-2}{n}}\int_M\pa{\tilde{S}-t\tilde{\mathcal{B}}\abs{\tilde{\boc}}_{\tilde{g}}^{-2}}dV_{\tilde{g}}\\
			&=\inf_{\tilde{g}\in[g]}\mathrm{Vol}_{\tilde{g}}(M)^{-\frac{n-2}{n}}\int_M\pa{2\tilde{R}-t\tilde{\mathcal{B}}\abs{\tilde{\boc}}_{\tilde{g}}^{-2}}dV_{\tilde{g}},
		\end{align*}
		where $\tilde{S}, \tilde{R},\tilde{\mathcal{B}}$ and $\tilde{\boc}$ are relative to the metric $\tilde{g}$. We remark that, given $g$, the geometric quantity $\mathcal{B}$ is cubic in $\boc$, hence, the term $\mathcal{B}\abs{\boc}^{-2}$ can be set equal to zero when the Bochner tensor is vanishing, implying that $\mathcal{B}\abs{\boc}^{-2}$ is continuous. Moreover, $$\nabla(\mathcal{B}\abs{\boc}^{-2})=(\nabla\mathcal{B})\abs{\boc}^{-2}-2\mathcal{B}\abs{\boc}^{-3}\nabla\abs{\boc}$$ 
		which is bounded by a constant multiple of $\abs{\nabla\boc}$. Hence, since $(M,g)$ is compact, we can conclude that hence $\mathcal{B}\abs{\boc}^{-2}$ is Lipschitz continuous.\\
		From now on, we will denote $\abs{\cdot}$, the norm with respect to the metric $g$ to simplify the notation.
		Let $f_{\eps}=\pa{\abs{\operatorname{B}}^2+\eps^2}^{q}$, for $\eps>0$ and $q=\frac{1}{2}\frac{m-1}{m+1}$, then, by \eqref{bw for b}, we obtain 
		\begin{align}\label{delta f eps}
			\Delta f_{\eps}
			&=4q(q-1)f_\eps^{1-\frac{2}{q}}\abs{\boc}^2\abs{\nabla\abs{\boc}}^2+qf^{1-\frac{1}{q}}_\eps\pa{2\abs{\nabla \boc}^2+\frac{4}{n}S\abs{\boc}^2-4\mathcal{B}}.
		\end{align}
		Since we are on a K\"{a}hler-Einstein manifold we have that
		\[\nabla\boc=\nabla\riem,\]
		thus, we can exploit \cite[Lemma 4.9]{Bando}, which gives us the following refined Kato inequality:
		\begin{align}\label{Kato Boc}
			\abs{\nabla\boc}^2\geq\frac{n+6}{n+2}\abs{\nabla\abs{\boc}}^2=\frac{m+3}{m+1}\abs{\nabla\abs{\boc}}^2.
		\end{align}
		Therefore, using \eqref{Kato Boc} and the fact that $f_\eps\geq\abs{\boc}^{2q}$, we get
		\begin{align*}
			\Delta f_{\eps}\geq 2q\pa{2(q-1)+\frac{m+3}{m+1}}\abs{\boc}^2\abs{\nabla\abs{\boc}}^2f_\eps^{1-\frac{2}{q}}+4qf_\eps^{1-\frac{1}{q}}\pa{\frac{R}{m}\abs{\boc}^2-\mathcal{B}},
		\end{align*}
		which, for our choice of $q$, is equivalent to
		\begin{align*}
			\Delta f_{\eps}\geq2\frac{m-1}{m+1}\pa{\frac{R}{m}\abs{\boc}^2-\mathcal{B}}f_\eps^{1-\frac{1}{q}}.
		\end{align*}
		Arguing as in \cite{BCDM}, we introduce the second order operator 
		\begin{align}\label{L beta}
			\mathcal{L}^\beta&=-\frac{4(n-1)}{n-2}\Delta+\beta S-\beta n\mathcal{B}\abs{\boc}^{-2}\\
			&=-\frac{2(2m-1)}{m-1}\Delta+2\beta R-2\beta m\mathcal{B}\abs{\boc}^{-2}.\notag
		\end{align} 
		and we see that
		\begin{align*}
			f_\eps\mathcal{L}^\beta f_\eps=&-\frac{2(2m-1)}{m-1}f_\eps\Delta f_\eps+2\beta R f_\eps^2-2\beta m\mathcal{B}\abs{\boc}^{-2}f_\eps^2\\
			\leq&-\frac{4(2m-1)}{m+1}\pa{\frac{R}{m}-\mathcal{B}\abs{\boc}^{-2}}\abs{\boc}^2f_\eps^{2-\frac{1}{q}}+2\beta Rf_\eps^2-2\beta m\mathcal{B}\abs{\boc}^{-2}f_\eps^2\\
			=&\pa{2\beta-\frac{4(2m-1)}{m+1}}\frac{R}{m}f_\eps^2+\pa{\frac{4(2m-1)}{m+1}-2\beta m}\mathcal{B}\abs{\boc}^{-2}f_\eps^2\\
			&+\eps^2\frac{4(2m-1)}{m+1}\pa{\frac{R}{m}-\mathcal{B}\abs{\boc}^{-2}}f_\eps^{2-\frac{1}{q}}.
		\end{align*}
		Taking $\beta=\ol{\beta}:=\frac{2(2m-1)}{m(m+1)}$, we have
		\[
		f_\eps\mathcal{L}^\beta f_\eps\leq\eps^2\frac{4(2m-1)}{m+1}\pa{\frac{R}{m}-\mathcal{B}\abs{\boc}^{-2}}f_\eps^{2-\frac{1}{q}}.
		\]
		Then, following \cite{BCDM}, we introduce a modification of the Yamabe invariant due to Bour and Carron (see \cite{BourCarron}) 
		\begin{align}\label{yamabe bc}
			\ol{Y}_g(\beta)&:=\inf_{u\neq0,\,u\in C^\infty_0}\frac{\int_Mu\mathcal{L}^\beta udV_g}{||u||_{L^\frac{2n}{n-2}}^2}\\
			&=\inf_{u\neq0,\,u\in C^\infty_0}\frac{\int_M\pa{\frac{2(2m-1)}{m-1}\abs{\nabla u}^2+2\beta R u^2-2\beta m\mathcal{B}\abs{\boc}^{-2}u^2}dV_g}{||u||_{L^\frac{2m}{m-1}}^2},\notag
		\end{align}
		for all $\beta\geq 0$. Note that, since $2-\frac{1}{q}<0$ and $f_\eps\geq\eps^{2q}$, we have
		\[
		0\leq f_\eps^{2-\frac{1}{q}}\eps^2\leq\eps^{4q}, \ra0\quad\text{ when}\quad\eps\ra0.
		\]
		Therefore, when $\beta=\ol{\beta}$ and $\eps\ra0$, 
		\begin{align*}
			\ol{Y}_g(\ol{\beta})&=
			\frac{\int_M f_\eps \mathcal{L}^{\ol{\beta}}f_\eps\,
				dV_g}{\pa{\int_M f_\eps^{\frac{2m}{m-1}}\,dV_g}^{\frac{m-1}{m}
			}}
			\leq\dfrac{4(2m-1)}{m+1}
			\dfrac{\int_M
				\pa{\frac{R}{m}-\mathcal{B}\abs{{\boc}}^{-2}}f_\eps^{2-\frac{1}{q}}\,
				dV_g}{\pa{\int_M f_\eps^{\frac{2m}{m-1}}\,dV_g}^{\frac{m-1}{m}}}
			\eps^2\ra0,
		\end{align*}
		that implies
		\[\ol{Y}_g(\ol{\beta})\leq0.\]
		Note that $\ol{Y}_g(1)=\ol{Y}^t(M,[g])$ for $t=n$. Moreover, since $M$ is compact $\ol{Y}_g(0)=0$ and, by definition, \eqref{yamabe bc} is the infimum of affine functions of $\beta$, thus it is concave and, for $\beta\in[0,1]$, we have
		\[(1-\beta)\ol{Y}_g(0)+\beta\ol{Y}_g(1)\leq\ol{Y}_g(\beta).\]
		We point out that for $m\geq2$, $\ol{\beta}\leq1$, and, since $\ol{Y}_g(\beta)$ is concave, we deduce
		\[\bar{\beta}\ol{Y}^n(M,[g])\leq\ol{Y}_g(\ol{\beta}),\]
		which implies 
		\[\ol{Y}^n(M,[g])\leq0.\]
		Following the same line of reasoning in \cite{GurskyLeBrun} (i.e., adapting an argument of \cite[Proposition 4.4]{LeeParker}), we know that there exists a unique metric $g'\in[g]$ such that 
		of 
		\[2R'-2m\mathcal{B}'\abs{\boc'}_{g'}^{-2}\]
		is a non-positive constant. In particular $g'=v^{\frac{2}{m-1}}g$, for $0<v\in C^{2,\alpha}(M)$, $\alpha\in(0,1)$. Arguing as in Proposition 3.5 of \cite{Gursky200}, and using the density of smooth metrics in the space of $C^{2,\alpha}$ metrics (with the $C^{2,\alpha}$ norm), there exists a unique $\hat{g}$ smooth such that 
		\begin{align}\label{hat g in}	\mathrm{Vol}_{\hat{g}}(M)^{-\frac{m-1}{m}}\int_M2\hat{R}dV_{\hat{g}}\leq\mathrm{Vol}_{\hat{g}}(M)^{-\frac{m-1}{m}}2m\int_M\hat{\mathcal{B}}\abs{\hat{\boc}}^{-2}_{\hat{g}}dV_{\hat{g}},
		\end{align}	
		where $\hat{R}$, $\hat{\mathcal{B}}$ and $\hat{\boc}$ are relative to the metric $\hat{g}$ .
		Therefore, since $g$ is K\"{a}hler-Einstein it attains the minimum of the usual Yamabe invariant $Y(M,[g])$ and due to the validity of \eqref{hat g in}, we have
		\begin{align*}
			Y(M,[g])&=\mathrm{Vol}_g(M)^{-\frac{m-1}{m}}\int_M2RdV_g\leq\mathrm{Vol}_{\hat{g}}(M)^{-\frac{m-1}{m}}\int_M2\hat{R}dV_{\hat{g}}\\
			&\leq\mathrm{Vol}_{\hat{g}}(M)^{-\frac{m-1}{m}}2m\int_M\hat{\mathcal{B}}\abs{\hat{\boc}}^{-2}_{\hat{g}}dV_{\hat{g}}.
		\end{align*}
		Using H\"{o}lder inequality, we get
		\begin{align*}
			\int_M\hat{\mathcal{B}}\abs{\hat{\boc}}^{-2}_{\hat{g}}dV_{\hat{g}}\leq\mathrm{Vol}(M)^{\frac{m-1}{m}}\pa{\int_M\abs{\hat{\mathcal{B}}}^{m}_{\hat{g}}\abs{\hat{\boc}}_{\hat{g}}^{-2m}dV_{\hat{g}}}^{\frac{1}{m}},
		\end{align*}
		that implies
		\begin{align*}
			Y(M,[g])\leq 2m\pa{\int_M\abs{\hat{\mathcal{B}}}_{\hat{g}}^{m}\abs{\hat{\boc}}^{-2m}_{\hat{g}}dV_{\hat{g}}}^{\frac{1}{m}}.
		\end{align*}
		Since the right-hand side is conformally invariant by definition, we obtain
		\begin{align*}
			Y(M,[g])\leq 2m\pa{\int_M\abs{{\mathcal{B}}}^{m}\abs{\boc}^{-2m}dV_g}^{\frac{1}{m}},
		\end{align*}
		which is \eqref{ineq yamabe and bochner}.\\
		If equality holds, then 
		\begin{align*}
			\mathrm{Vol}_{\hat{g}}(M)^{-\frac{m-1}{m}}\int_M\hat{S}= 2m\pa{\int_M\abs{{\hat{\mathcal{B}}}}^{m}_{\hat{g}}\abs{\hat{\boc}}^{-2m}_{\hat{g}}dV_{\hat{g}}}^{\frac{1}{m}},
		\end{align*}
		all previous inequalities are equalities and by definition $\ol{Y}^n(M,[g])=0$. Furthermore, since
		\[
		Y(M,[g])=\mathrm{Vol}_{\hat{g}}(M)^{-\frac{m-1}{m}}
		\int_M 2\hat{R}\, dV_{\hat{g}},
		\]
		we observe that $\hat{g}$ attains the minimum of the usual Yamabe invariant and,
		therefore,
		it is a solution of the Yamabe problem in $[g]$,
		which implies that $\hat{S}$ is constant.
		Since $g$ is an Einstein
		metric in $[g]$,
		we can exploit a well-known result due to Obata (\cite{obata})
		in order to conclude that $\hat{g}=g$ and, as a consequence,
		$\hat{R}=R$, $\hat{\mathcal{B}}=\mathcal{B}$ and $\hat{\boc}=\boc$.
		This implies that
		\begin{equation} \label{SB}
			R-m\mathcal{B}\abs{\boc}^{-2}=0 \Longra
			\mathcal{B}=\dfrac{R}{m}\abs{\boc}^2;
		\end{equation}
		therefore, integrating the Bochner-Weitzenb\"{o}ck formula, we get
		\[
		0=\int_M\pa{\abs{\nabla\boc}^2+\dfrac{2}{m}R\abs{\boc}^2-2\mathcal{B}\,} dV_g=
		\int_M\abs{\nabla\boc}^2dV_g,
		\]
		which implies that $\abs{\nabla\boc}^2\equiv 0$ on $M$, i.e.
		$(M,g)$ is locally symmetric. The converse is a direct consequence of the Bochner-Weitzenb\"{o}ck formula.
	\end{proof}
	Combining Theorem \ref{thm bochner intro} and \eqref{estimate mathcal B} (for $m\geq3$), we obtain the following isolation result for the Bochner tensor of a K\"{a}hler-Einstein metric, improving the result of Theorem 4.2 in \cite{ChongDongLinRen} for dimension six and recovering it in higher dimensions. In dimension $m=2$ we use a different estimate, which relies on the fact that in this dimension $\boc$ coincides with $\weyl^+$.
	\begin{cor}\label{c-boc}
		Let $(M,g,J)$ be a compact (conformally) K\"{a}hler-Einstein manifold of dimension $n=2m\geq4$ with positive Yamabe invariant. Then either $(M,g,J)$ is biholomorphically homothetic to the complex projective space $\mathbb{CP}^m$endowed with the Fubini-Study metric or, if $\operatorname{B}\neq0$, 
		\begin{align}\label{yamabe and bochner corol}
			Y(M,[g])\leq C(m)||\boc||_{L^{\frac{n}{2}}},
		\end{align}
		where $C(2)=\sqrt{6}$ and
		$$C(m)=\frac{3(m^2-2)}{\sqrt{m^2-1}}$$
		for $m\geq3$.
	\end{cor}
	\noindent
	Note that when $n=4$ the Bochner tensor coincide with the self-dual part of the Weyl tensor, hence the constant $C(2)$ is not obtained by \eqref{estimate mathcal B} but by combining the peculiarities of dimension four (see also equation \eqref{estimate on Q} and \cite{Gursky200,Huisken}) with Theorem \ref{thm bochner intro}. In this way, we recover the optimal pinching result obtained for $\weyl^+$ in dimension $4$ by Gursky and LeBrun in \cite{GurskyLeBrun, Gursky200} and which was recovered by Itoh and Kobayashi in \cite{ItohKobayashi} studying isolation results for the Bochner tensor. On the other hand, $C(m)$ is obtained combining \eqref{estimate mathcal B} with Theorem \ref{thm bochner intro}.
	\subsection{Sasaki-Einstein manifolds}
	\begin{defi}
		We say that $(M,\phi,\xi,\eta,g)$ is a Sasakian manifold of dimension $n=2m+1\geq5$ if $g,\eta,\xi$ and $\phi$ are respectively a Riemannian metric, a $1$-form, a unit Killing vector field (known as Reeb vector field) and a $(1,1)$-tensor field satisfying
		\begin{align*}
			\eta(X)&=g(\xi,X);\\
			\pa{\nabla_X\eta}(Y)&=g(X,\phi Y);\\
			\phi^2X&=-X+\eta(X)\xi;\\
			\pa{\nabla_X\phi}(Y)&=g(X,Y)\xi-\eta(Y)X,
		\end{align*}
	for every $X,Y\in\mathfrak{X}(M)$.\\
	 $(M,\phi,\xi,\eta,g)$ is said to $\eta$-Einstein if the Ricci tensor has the form
	\begin{align*}
		\ric=\lambda g+\mu\eta\otimes\eta,
	\end{align*}
	where
	\begin{align*}
		\lambda=\frac{S}{n-1}-1, \quad\mu=-\frac{S}{n-1}+n,
	\end{align*}
	moreover, when $(M,\phi,\xi,\eta,g)$ is $\eta$-Einstein the scalar curvature is constant.
	\end{defi}
	Similarly to the K\"{a}hlerian case, an analogous of the Weyl tensor has been intriduced also in the Sasakian setting, namely the contact Bochner tensor. 
	\begin{defi}
		Let $(M,\phi,\xi,\eta,g)$ be a Sasaki manifold of dimension $n=2m+1\geq 5$, then the contact Bochner tensor is the $(0,4)$-tensor whose components, with respect to a local orthonormal frame, are
		\begin{align}\label{contact Bochner}
			B_{ijst}=&R_{ijst}-\frac{1}{n+3}[R_{is}\delta_{jt}+R_{jt}\delta_{is}-R_{it}\delta_{js}-R_{js}\delta_{it}\\
			&+R_{ir}\phi^r_s\phi_{jt}+R_{jr}\phi^r_t\phi_{is}-R_{ir}\phi^r_t\phi_{js}-R_{jr}\phi^r_s\phi_{it}+2R_{ir}\phi^r_j\phi_{st}+2\phi_{ij}R_{sr}\phi^r_t\notag\\
			&-R_{is}\eta_j\eta_t-R_{jt}\eta_i\eta_s+R_{it}\eta_j\eta_s+R_{js}\eta_i\eta_t]+\frac{k+n-1}{n+3}[\phi_{is}\phi_{jt}-\phi_{it}\phi_{js}+2\phi_{ij}\phi_{st}]\notag\\
			&\frac{k-4}{n+3}[\delta_{is}\delta_{jt}-\delta_{it}\delta_{js}]-\frac{k}{n+3}[\delta_{is}\eta_j\eta_t+\delta_{jt}\eta_i\eta_s-\delta_{it}\eta_j\eta_s-\delta_{js}\eta_i\eta_t],\notag
		\end{align}
		where $k=\frac{S+n-1}{n+1}$. 
	\end{defi}
	Note that 
	\begin{align*}
		B_{ijst}=&R_{ijst}-\frac{k}{n-1}\sq{\delta_{is}\delta_{jt}-\delta_{it}\delta_{js}}\\
		&-\pa{\frac{k}{n-1}-1}\sq{\phi_{is}\phi_{jt}-\phi_{it}\phi_{js}+2\phi_{ij}\phi_{st}-\delta_{is}\eta_j\eta_t-\delta_{jt}\eta_i\eta_s+\delta_{it}\eta_j\eta_s+\delta_{js}\eta_i\eta_t}
	\end{align*}
	if and only if $(M,\phi,\xi,\eta,g)$ is $\eta$-Einstein.
	Given a Sasakian manifold $(M,\phi,\xi,\eta,g)$, with K\"{a}hler cone metric $\ol{g}=dr^2+r^2g$ we deform $r$ into $r'=r^c$, where $c$ is a positive constant. This deformation transforms the Sasakian structure into a new one defined by
	\begin{align*}
		\phi_c=\phi;\quad\xi_c=c^{-1}\xi;\quad\eta_c=c\eta;\quad g_c=cg+c(c-1)\eta\otimes\eta.
	\end{align*}
	Such a transformation is known as $D$-homothetic deformation. As a consequence the Riemann curvature tensor, the Ricci tensor and the scalar curvature associated to $g_c$ are
	\begin{align*}
		\riem_c&=c\riem+c(c^2-1)(\eta\otimes\eta)\KN g-c(c-1)\Phi;\\
		\ric_c&=\ric-2(c-1)g+(c-1)\sq{(n-1)c+n+1}\eta\otimes\eta;\\
		S_c&=c^{-1}S-c^{-1}(c-1)(n-1),
	\end{align*}
	respectively, where $\Phi_{ijst}=\phi_{is}\phi_{jt}-\phi_{it}\phi_{js}+2\phi_{ij}\phi_{st}.$ Furthermore, the volume form transforms as $dV_{g_c}=c^{\frac{n+1}{2}}dV_{g}$.\\
	It has been proved by Itoh and Kobayashi (see section $3.2$ of \cite{ItohKobayashi}) that a Sasakian $\eta$-Einstein manifold can always be deformed into a Sasaki-Einstein manifold whose contact Bochner tensor coincides with the Weyl tensor of the Einstein metric. As a consequence, we have the validity of Theorem \ref{thm contact Bochner}, which improves the already known result by Itoh and Kobayashi (\cite[Theorem B]{ItohKobayashi}).
	\begin{proof}[Proof of Theorem \ref{thm contact Bochner}]
		Given a Sasaki $\eta$-Einstein manifold of dimension $n=2m+1\geq5$ with scalar curvature $S>-(n-1)$, we can consider the $D$-homothetic transformation $(\phi,\xi,\eta,g)\ra(\phi_{\alpha},\xi_{\alpha},\eta_\alpha,g_\alpha)$, for $$\alpha=\frac{S+n-1}{(n-1)(n+1)},$$
		then, by \cite{ItohKobayashi}, we have that the metric $g_\alpha$ is Einstein with
		\begin{align*}
			\ric_\alpha&=(n-1)g_\alpha;\\
			S_\alpha&=n(n-1);\\
			\operatorname{B}_\alpha&=\weyl_{\alpha};
		\end{align*}
		and
		\begin{align}\label{weyl norm}
			||\weyl_\alpha||_{L^{\frac{n}{2}}(g_{\alpha})}=||\operatorname{B}_\alpha||_{L^{\frac{n}{2}}(g_{\alpha})}=\alpha^{\frac{1}{n}}||\operatorname{B}||_{L^{\frac{n}{2}}(g)}.
		\end{align}
		It was proved in \cite{BCDM} that, given an Einstein manifold of dimension $n\geq5$, then if $n\neq5$ and $\weyl_{\alpha}\neq0$,
		\begin{align}\label{n=5 contact bochner}
			Y(M,[g_{\alpha}])\leq n\pa{\int_M \abs{Q}^{\frac{n}{2}}\abs{\weyl}^{-n}}^{\frac{2}{n}}\leq nC(n)||\weyl_{\alpha}||_{L^\frac{n}{2}(g_{\alpha})},
		\end{align}
		where $\abs{\cdot}$ denotes the norm with respect to the metric $g$ and $C(n)=\pa{2+\frac{1}{2}\frac{n(n-1)-4}{\sqrt{n(n-2)(n-1)(n+1)}}}$.\\
		If $n=5$ and $\weyl_{\alpha}\neq 0$,
		\begin{align}\label{n grater thatn 5}
			Y(M,[g_{\alpha}])< \frac{16}{3}\pa{\int_M \abs{Q}^{\frac{5}{2}}\abs{\weyl}^{-5}}^{\frac{2}{5}}\leq \frac{64}{3\sqrt{10}}||\weyl_{\alpha}||_{L^\frac{5}{2}(g_{\alpha})}.
		\end{align}
		Moreover,
		\[
		Y(M,[g_{\alpha}])=\mathrm{Vol}_{g_\alpha}(M)^{\frac{2}{n}}S_{\alpha}=\mathrm{Vol}_g(M)^{\frac{2}{n}}n(n-1)\alpha^{\frac{n+1}{n}}.
		\]
		Combining \eqref{weyl norm}, \eqref{n=5 contact bochner} and \eqref{n grater thatn 5}, we deduce that when $n=5$
		\begin{align*}
			||\operatorname{B}||_{L^{\frac{5}{2}}(g)}>\mathrm{Vol}_g(M)^{\frac{2}{5}}\frac{15\sqrt{10}}{16}\alpha
		\end{align*}
		and when $n>5$
		\begin{align*}
			||\operatorname{B}||_{L^{\frac{n}{2}}(g)}\geq\mathrm{Vol}_g(M)^{\frac{2}{n}}C(n)^{-1}(n-1)\alpha;
		\end{align*}
		thus for $\alpha=\frac{S+n-1}{(n+1)(n-1)}$, we have the claim.
	\end{proof}

	\bibliographystyle{abbrv}
	\bibliography{bibliography}
\end{document}